%% file: ARXIV-V2-MindtLangDomschke2019.tex
\documentclass[11pt]{article}

\usepackage{amssymb}
\usepackage{amsmath}
\usepackage{amsthm}
\usepackage{graphicx}
\usepackage{accents}
\usepackage[margin=0.65cm]{caption}
\usepackage[caption=false]{subfig}
\usepackage{enumerate}

\usepackage{tikz}
\usetikzlibrary{fit,calc}
\usetikzlibrary{arrows,shapes,positioning,shadows,trees}
\usetikzlibrary{decorations.pathreplacing}
\usetikzlibrary{arrows.meta}

\tikzset{
basic/.style  = {draw, text width=1.5cm, drop shadow, font=\sffamily, rectangle},
  root/.style   = {basic, rounded corners=1pt, thin, align=center,
                   fill=blue!5},
  level 2/.style = {basic, rounded corners=1pt, thin,align=center, fill=white!60,
                   text width=8em},
  level 3/.style = {basic, thin, align=left, fill=pink!60, text width=8em}
}

\newcommand{\R}{\mathbb R}
\newcommand{\I}{\mathbb I}
\newcommand{\N}{\mathbb N}
\newcommand{\Ro}{\accentset{\circ}{\mathbb R}}
\newcommand{\bPi}{\bar{\Pi}}
\newcommand{\bY}{\bar{Y}}
\newcommand{\bU}{\bar{U}}
\newcommand{\bdelta}{\bar{\delta}}
\newcommand{\brho}{\bar{\rho}}
\newcommand{\bp}{\bar{p}}
\newcommand{\bc}{\bar{c}}
\newcommand{\bs}{\bar{s}}
\newcommand{\bq}{\bar{q}}
\newcommand{\tU}{\tilde{U}}
\newcommand{\hU}{\hat{U}}
\newcommand{\hK}{\hat{K}}
\newcommand{\cU}{\mathcal{U}}
\newcommand{\LX}{\mathcal{L}}
\newcommand{\cE}{\mathcal{E}}
\newcommand{\cR}{\mathcal{R}}
\newcommand{\cT}{\mathcal{T}}
\newcommand{\cM}{\mathcal{M}}
\newcommand{\cA}{\mathcal{A}}

\newcommand{\cJ}{\mathcal{J}}

\newcommand{\bfq}{{\bf q}}
\newcommand{\bfL}{{\bf L}}
\newcommand{\bfC}{{\bf C}}
\newcommand{\bfBV}{{\bf BV}}
\newcommand{\ds}{\displaystyle}

\newtheorem{theorem}{Theorem}[section]

\newtheorem{definition}{Definition}[section]

\newcommand{\abstand}{\vspace{\baselineskip}}

\usepackage{xcolor,todonotes}

\author{P.~Mindt, J.~Lang, P.~Domschke}
\title{Entropy-Preserving Coupling of Hierarchical Gas Models}
\author{
Pascal Mindt, Jens Lang\footnote{corresponding author} , and Pia Domschke
\\[0.5cm]
{\small \it Technische Universit\"at Darmstadt} \\
{\small \it Dolivostra{\ss}e 15, 64293 Darmstadt, Germany}\\
{\small mindt@mathematik.tu-darmstadt.de}\\
{\small lang@mathematik.tu-darmstadt.de}\\
{\small domschke@mathematik.tu-darmstadt.de}
}
\date{December 14, 2019}

\begin{document}

\maketitle

\begin{abstract}
This paper is concerned with coupling conditions at junctions
for transport models which differ in their fidelity to describe
transient flow in gas pipelines. It also includes the integration
of compressors between two pipes with possibly different models.
A hierarchy of three one-dimensional gas transport models is built
through the $3\times 3$ polytropic Euler equations, the $2\times 2$
isentropic Euler equations and a simplified version of it for
small velocities. To ensure entropy preservation, we make use of the
novel entropy-preserving coupling conditions recently proposed by
{\sc Lang} and {\sc Mindt} [Netw.~Heterog.~Media, 13:177-190, 2018]
and require the continuity of the total enthalpy at the junction and
that the specific entropy for pipes with outgoing flow equals the
convex combination of all entropies that belong to pipes with incoming
flow. We prove the existence and uniqueness of solutions to generalised
Riemann problems at a junction in the neighbourhood of constant coupling
functions and stationary states which belong to the subsonic region. This
provides the basis for the well-posedness of certain Cauchy problems for
initial data with sufficiently small total variation.
\end{abstract}

\abstand


\section{Introduction}
The transient flow of natural gas through pipeline networks in a
dynamic supply-demand environment has been attracting increasing
interest. Such distribution networks play an important role in
future energy systems. They also allow the storage of renewable
electric energy within a power-to-gas process chain. Simulation
and optimisation of gas pipeline networks require the study of
large scale models ranging from complex compressor stations to
networks of a whole country. There exists a bunch of models based
on the compressible Euler equations to predict the network behaviour
with varying accuracy, see e.g. \cite{BandaHertyKlar2006a,Osiadacz1996,OsiadaczChaczykowski2001}
and the nice overview in \cite{BrouwerGasserHerty2011}.

Since more accurate models are computationally more expensive, an
appropriate use of a hierarchy of models is desirable. In a sequence
of papers \cite{DomschkeDuaStolwijkLangMehrmann2018,DomschkeKolbLang2011,DomschkeKolbLang2015},
we have developed adaptive strategies to automatically control
the model selection, mainly depending on the dynamics of the gas flow.
Generally, simplified models can be applied in regions with low activity,
while sophisticated models have to be used in regions, where the dynamical
behaviour has to be resolved in more detail.

A crucial point in the one-dimensional modelling process of gas networks is
the determination of physically sound coupling conditions at junctions. Beside
the natural mass and energy conservation, the equality of the dynamic pressure
\cite{ColomboMauri2008} or the pressure itself \cite{BandaHertyKlar2006b,Herty2008}
are widely used in the literature. The latter one is the usual choice in the engineering community. For isothermal and isentropic flows, investigations in \cite{Reigstad2014,Reigstad2015}
showed that both pressure-based coupling conditions can deliver non-physical solutions characterized by the production of mechanical energy at a junction and should be
replaced by the equality of enthalpy. Recently, we have extended this result to
$3\times 3$ Euler systems with source terms at a junction of pipes with possibly different cross-sectional areas \cite{LangMindt2018}. We additionally propose entropy-preserving coupling conditions, i.e.,
we require that the specific entropy for pipes with outgoing flow equals the convex combination of all entropies that belong to pipes with incoming flow.

In this paper, we generalize the design of entropy-preserving coupling conditions in
order to account for varying models at a single junction. A hierarchy of three one-dimensional gas transport models which differ in their fidelity is built through the $3\!\times\!3$ polytropic Euler equations, the $2\!\times\!2$ isentropic Euler equations and a simplified version of it, where the kinetic energy is neglected. We also consider the practically important case of a compressor connected by two pipes with possibly different
gas transport models. We first define solutions of generalised one-sided Riemann
problems at a junction and show then by suitable application of the
Implicit Function Theorem the existence and uniqueness of such solutions
in the neighbourhood of constant coupling functions
and stationary states which belong to the subsonic region. This
provides the basis for the well-posedness of certain Cauchy problems for
initial data with sufficiently small total variation.

The paper is organised as follows. In Sect.~\ref{sec:ModelHierarchy},
we introduce a model hierarchy
for polytropic Euler equations. Thermodynamically consistent coupling
conditions are described in Sect.~\ref{sec:ModelCoupling}, including
coupling at junctions and two models of compressor coupling.
In Sect.~\ref{sec:GeneralRiemannProblem}, we study the solvability of
generalized Riemann problems. The corresponding Cauchy problems and their solutions
are studied in Sect.~\ref{sec:CauchyProblemJunction}.

\section{Model Hierarchy for Polytropic Euler Equations} \label{sec:ModelHierarchy}
We consider the one-dimensional polytropic Euler equations with source terms
as our most accurate model~$\cM_1$ to describe the gas flow in a pipe of
infinite length,
\begin{eqnarray}
\label{euler_eqs}
\partial_t U + \partial_x F_{\cM_1}(U)
&=& G_{\cM_1}(x,t,U),\quad (x,t)\in\R\times\R^+,\\
\label{euler_init}
U(x,0) &=& U_0(x), \quad x\in \R,
\end{eqnarray}
with thermodynamic variables and flux functions
\begin{equation}
\label{euler_flux}
U=
\begin{pmatrix}
\rho\\[1mm]
\rho u\\[1mm]
E
\end{pmatrix}\quad\mbox{and}\quad
F_{\cM_1}(U)=
\begin{pmatrix}
\rho u\\[1mm]
\rho u^2+p\\[1mm]
u(E+p)
\end{pmatrix}.
\end{equation}
Here, $\rho$ is the density, $u$ is the velocity, $p$ is the pressure,
and $E$ is the total energy. Let~$c_v>0$ and $c_p>0$ be the specific heat
at constant volume and pressure, respectively. Then, $R=c_p-c_v$ is the
gas constant and $\gamma=c_p/c_v>1$ is the adiabatic exponent. The
relation between the specific internal energy~$e$ and the temperature~$T$ of a
polytropic gas is described by~$e=c_vT$. Together with the total
energy~$E=\rho e+\rho u^2/2$ and the ideal gas law~$p=\rho RT$,
the equation of state for an {\it ideal polytropic gas} in the common form reads
\begin{equation}
p = (\gamma -1)\left( E-\frac12\rho u^2\right).
\end{equation}
For later use, we introduce the mass flux~$q=\rho u$,
the specific entropy~$s$, the total enthalpy~$h$, and
the speed of sound~$c$ defined by
\begin{equation}
\label{euler_shc}
s = c_v\,\ln\left( \frac{p}{\rho^\gamma}\right)+s_0,\quad
h = \frac{E+p}{\rho},\quad\text{and}\quad
c=\sqrt{\left.\frac{\partial p}{\partial \rho}\right|_s}
=\sqrt{\frac{\gamma p}{\rho}},
\end{equation}
where $s_0\ge 0$ is a constant entropy value. More details about the underlying
thermodynamic principles can be found, e.g., in \cite[Sect. 14.4]{LeVeque2002}. The
right-hand side vector~$G_{\cM_1}(x,t,U)$ describes source terms, e.g., gravity and friction.

A first simplification for small disturbances around some background state is the
use of an {\it isentropic flow}, where the entropy $s$ is taken as constant throughout
the gas. In this case, we can drop the third equation in \eqref{euler_eqs}, i.e., the
conservation of energy. The isentropic Euler equations taken as model $\cM_2$ are
\begin{eqnarray}
\label{iso_euler_eqs}
\partial_t U + \partial_x F_{\cM_2}(U)
&=& G_{\cM_2}(x,t,U),\quad (x,t)\in\R\times\R^+,\\
\label{iso-euler_init}
U(x,0) &=& U_0(x), \quad x\in \R,
\end{eqnarray}
with thermodynamic variables and flux functions
\begin{equation}
\label{iso_euler_flux}
U=
\begin{pmatrix}
\rho\\[1mm]
\rho u
\end{pmatrix}\quad\mbox{and}\quad
F_{\cM_2}(U)=
\begin{pmatrix}
\rho u\\[1mm]
\rho u^2+p(\rho)
\end{pmatrix}.
\end{equation}
Taking $s=\bs$ in \eqref{euler_shc}, we get an explicit relation
between pressure and density,
\begin{equation}
p(\rho) = \kappa\rho^\gamma\quad\text{with}\quad
\kappa=e^{\frac{\bs-s_0}{c_v}},
\end{equation}
which serves now as equation of state for the isentropic Euler equations.
Total energy, total entropy and speed of sound simplify to functions
of $\rho$ and $u$,
\begin{equation}
\label{iso_euler_ehc}
E = \frac{\kappa\rho^\gamma}{\gamma-1}+\frac12\rho u^2,\quad
h = \frac{\kappa\gamma}{\gamma-1}\rho^{\gamma-1}+\frac12 u^2\quad
\text{and}\quad c=\sqrt{\kappa\gamma\rho^{\gamma-1}}.
\end{equation}
The isentropic equations are still nonlinear and shocks
can appear, if we allow arbitrary data. Then entropy
and energy will jump to a higher level across the shock,
indicating the correct vanishing-viscosity solution. Although
conservation of energy is no longer satisfied, such isentropic
shocks may be a good approximation to reality, if they are weak
enough. Further arguments are given in \cite[Sect. 14.5]{LeVeque2002}.

In many practical situations, the spatial derivative of
the kinetic energy $\partial_x(\rho u^2)$ can be neglected \cite{Osiadacz1996}.
This yields model $\cM_3$ -- a
further simplification of \eqref{iso_euler_eqs} with
thermodynamic variables and flux functions
\begin{equation}
\label{m3-euler_flux}
U=
\begin{pmatrix}
\rho\\[1mm]
\rho u
\end{pmatrix}\quad\mbox{and}\quad
F_{\cM_3}(U)=
\begin{pmatrix}
\rho u\\[1mm]
p(\rho)
\end{pmatrix}.
\end{equation}
We formally set $G_{\cM_3}= G_{\cM_2}$ and use for the
total energy and enthalpy the following approximations:
\begin{equation}
E = \frac{\kappa\rho^\gamma}{\gamma-1}\quad
\text{and}\quad h = \frac{\kappa\gamma}{\gamma-1}\rho^{\gamma-1}.
\end{equation}
The speed of sound in \eqref{iso_euler_ehc} remains unchanged.

The models $\cM_i$, $i=1,2,3,$ define a hierarchy of models with
decreasing fidelity. Their characteristic eigenvalues are
given by
\begin{eqnarray}
& \lambda_1^{\cM_1}(U) = u-c,\quad
\lambda_2^{\cM_1}(U) = u,\quad
\lambda_3^{\cM_1}(U) = u+c, \\[2mm]
& \lambda_1^{\cM_2}(U) = u-c,\quad
\lambda_2^{\cM_2}(U) = u+c, \\[2mm]
& \lambda_1^{\cM_3}(U) = -c,\quad
\lambda_2^{\cM_3}(U) = c.
\end{eqnarray}
In what follows, we will work within the subsonic region, i.e.,
$|u|<c$. In this case, only $\lambda_2^{\cM_1}(U)$ can
change its sign, depending on the velocity. It will be
always clear to which model the state $U$ belongs.

\section{Thermodynamically Consistent Model Coupling}\label{sec:ModelCoupling}
\subsection{Coupling at Junctions}
In this section, we consider one-dimensional gas flow on a network
consisting of a single junction connecting $N$ pipe sections of
infinite length
\begin{eqnarray}
\label{network_euler_eqs}
\partial_t U^{(j)} + \partial_xF_{m_j}(U^{(j)}) &=& G_{m_j}(x,t,U^{(j)})
,\quad (x,t)\in\R^+\times\R^+,\\
\label{network_euler_init}
U^{(j)}(x,0) &=& U^{(j)}_0(x), \quad x\in \R,
\end{eqnarray}
for $j=1,\ldots,N$. The possibly different models are identified by
the parameters $m_j\in\cM:=\{\cM_1,\cM_2,\cM_3\}$. Each pipe is described
by a vector, $\nu_i\in\R^3\setminus\{0\}$, originating from the
common junction and parameterized by $x\in\R^+$, the real halfline~$[0,\infty)$.
The surface section of the pipe equals $\|\nu_i\|\!\ne\!0$.
We assume $\nu_i\ne \nu_j$ for $i\ne j$.

For each model, we introduce two sets of subsonic data
\begin{eqnarray}
D_{+}^{\cM_1} &\!=\!& \{ U=(\rho,\rho u,E)\in\Ro^+\times\R\times\Ro^+:
\,\lambda_1^{\cM_1}(U)<0<u<\lambda_3^{\cM_1}(U)\},\\
D_{-}^{\cM_1} &\!=\!& \{ U=(\rho,\rho u,E)\in\Ro^+\times\R\times\Ro^+:
\,\lambda_1^{\cM_1}(U)<u<0<\lambda_3^{\cM_1}(U)\},\\
D_{+}^{\cM_i} &\!=\!& \{ U=(\rho,\rho u)\in\Ro^+\times\R:
\,\lambda_1^{\cM_i}(U)<0<u<\lambda_2^{\cM_i}(U)\},\; i=2,3,\\
D_{-}^{\cM_i} &\!=\!& \{ U=(\rho,\rho u)\in\Ro^+\times\R:
\,\lambda_1^{\cM_i}(U)<u<0<\lambda_2^{\cM_i}(U)\},\; i=2,3.
\end{eqnarray}
with $\Ro^+=(0,\infty)$. Due to $\lambda_2^{\cM_1}(U)=u$ and the orientation of
the pipes, we can relate pipes modelled by $\cM_1$ with a flow direction
towards the junction with~$D_{-}^{\cM_1}$ (incoming flow), while~$D_{+}^{\cM_1}$
corresponds to pipes with flow direction away from the junction (outgoing flow).
Since $\lambda_1^{\cM_i}(U)<0$ and $\lambda_2^{\cM_i}(U)>0$ for the isentropic
models~$\cM_i$, $i=2,3$, a distinction between incoming and outgoing pipes is
usually not necessary. However, this separation becomes crucial if these models
are coupled with~$\cM_1$. This point is discussed in more detail in Section
\ref{sec:GeneralRiemannProblem}.

The corresponding index sets of the incoming and outgoing pipes are defined by
\begin{equation}
\I_{i}^{\cM_k}:=\{j:U^{(j)}\in D_{-}^{\cM_k}\}\quad\text{and}\quad
\I_{o}^{\cM_k}:=\{j:U^{(j)}\in D_{+}^{\cM_k}\},\;k=1,2,3.
\end{equation}
For later use, we define the index sets
$\I_i:=\I_i^{\cM_1}\cup\I_i^{\cM_2}\cup\I_i^{\cM_3}$ and
$\I_o:=\I_o^{\cM_1}\cup\I_o^{\cM_2}\cup\I_o^{\cM_3}$ for incoming
and outgoing pipes, and the special numbers
$N_o^{\cM_k}:=\dim\left(\I_o^{\cM_k}\right)$ and
$N^{\cM_k}:=\dim\left(\I_o^{\cM_k}\cup\I_i^{\cM_k}\right)$.
We will only consider cases with
\begin{equation}
\I_i\cup\I_o=\{1,\ldots,N\}.
\end{equation}
The coupling of the different model equations at the junction-pipe
interface is prescribed by a set of coupling conditions of the form
\begin{equation}
\label{euler_ccond}
\Phi \left( U^{(1)}(0^+,t),\ldots,U^{(N)}(0^+,t) \right) = \Pi(t),
\end{equation}
where $\Phi$ is a possibly nonlinear function of the traces
$U^{(j)}(0^+,t)=\lim_{x\rightarrow 0^+}U^{(j)}(x,t)$ of
the unknown variables and $\Pi$ is a coupling constant, which depends
only on time. We will use the entropy-preserving coupling conditions
from \cite{LangMindt2018} for $t>0$,
\begin{align}
\tag{M} \sum_{j=1}^{N} \|\nu_j\| q_j(0^+,t) &= 0, &&\text{(mass conservation),} \label{equ:M}\\
\tag{H} h_j(0^+,t) &= h^*(t), \quad j=1,\ldots,N, &&\text{(equality of enthalpy),} \label{equ:H}\\[3mm]
\tag{S} s_j(0^+,t) &= s^*(t), \quad j\in\I_o  &&\text{(equality of outgoing entropy)} \label{equ:S}
\end{align}
with the entropy mix
\begin{equation}
s^*(t) = \frac{1}{\sum_{j\in\I_i}\,\|\nu_j\|q_j(0^+,t)}
\,\sum_{j\in\I_i}\,\|\nu_j\|(q_js_j)(0^+,t).
\end{equation}
The function $h^*(t)$ in \eqref{equ:H} 
is not prescribed and determined by the flow itself.
It must be eliminated before the coupling function $\Phi$ is defined.
We fix one of the enthalpy equations for a certain $j=j_0$,
which will be specified later, and subtract
from it all the other ones. This gives the coupling conditions
\begin{equation}
0 = \Phi \left( U^{(1)}(0^+,t),\ldots,U^{(N)}(0^+,t) \right) =
\begin{pmatrix}
\sum_{j=1}^{N} \|\nu_j\| q_j(0^+,t)\\[2mm]
h_{j_0}(0^+,t) - h_1(0^+,t)\\
\cdots\\
h_{j_0}(0^+,t) - h_{j_0-1}(0^+,t)\\[2mm]
h_{j_0}(0^+,t) - h_{j_0+1}(0^+,t)\\
\cdots\\
h_{j_0}(0^+,t) - h_N(0^+,t)\\[2mm]
s_{j_1}(0^+,t) - s^*(t)\\
\cdots\\
s_{j_{N_o}}(0^+,t) - s^*(t)
\end{pmatrix}
\end{equation}
with $N_o=\dim(\I_o)$. Note that here $\Pi(t)\equiv 0$.

\subsection{Compressor Coupling}
Compressors in a network are typically placed between two pipes
with equal surface section and have
to be described by special coupling conditions. The task of a compressor
is to increase the pressure which is permanently decreased through friction.
We consider the resulting compression under adiabatic conditions, i.e.,
zero heat transfer between the gas and the surroundings, and as reversible process in
which the entropy remains constant. This leads to the following
coupling conditions (see also \cite[Chapt. 4.4]{Menon2005}):
\begin{align}
\tag{CM} q_1(0^+,t) + q_2(0^+,t) &= 0, && \text{(mass conservation),}\label{equ:CM}\\
\tag{CP1} \frac{\gamma}{\gamma-1}RT_1(0^+,t) \left(\left(\ds\frac{p_2(0^+,t)}{p_1(0^+,t)}
       \right)^{\frac{\gamma-1}{\gamma}}-1\right) &= H^*(t), && \text{(increase of pressure),} \label{equ:CP1}\\
\tag{CS} s_1(0^+,t) - s_2(0^+,t) &= 0 && \text{(equality of entropy).} \label{equ:CS}
\end{align}
Here, $U^{(1)}\in D_{-}^{\cM_i}$ and $U^{(2)}\in D_{+}^{\cM_j}$, hence
different models for the two pipes are allowed. The
coupling constant $H^*(t)$ stands for the change in adiabatic enthalpy,
necessary to raise the incoming pressure $p_1(0^+,t)$ to the outgoing
pressure $p_2(0^+,t)$. The condition \eqref{equ:CS} can be also expressed in
the form
\begin{equation}
\ds\frac{T_2(0^+,t)}{T_1(0^+,t)} =
\left(\frac{p_2(0^+,t)}{p_1(0^+,t)}\right)^{\frac{\gamma-1}{\gamma}}.
\end{equation}
In optimal control problems, $H^*(t)$ is often replaced by the theoretical
compressor power, $P^*(t)=C_p\,qH^*(t),\;C_p=const.$
\cite{EhrhardtSteinbach2005,Menon2005}, which can
be also used as coupling constant. In this case, we have
\begin{align}
\tag{CP2} \frac{\gamma}{\gamma-1}C_pRq_2(0^+,t)T_1(0^+,t)
\left(\left(\ds\frac{p_2(0^+,t)}{p_1(0^+,t)}
\right)^{\frac{\gamma-1}{\gamma}}-1\right) &= P^*(t). \label{equ:CP2}
\end{align}
Note that $q_2(0^+,t)>0$.

The coupling conditions defined above yield the two coupling functions
\begin{equation}
\Phi_1 \left( U^{(1)}(0^+,t),U^{(2)}(0^+,t) \right) =
\begin{pmatrix}
q_1(0^+,t) + q_2(0^+,t)\\[1mm]
\ds\frac{\gamma R}{\gamma-1}T_1(0^+,t) \left(
\left(\ds\frac{p_2(0^+,t)}{p_1(0^+,t)}
\right)^{\frac{\gamma-1}{\gamma}}-1\right) \\[2mm]
s_1(0^+,t) - s_2(0^+,t)
\end{pmatrix},
\end{equation}
\begin{equation}
\Phi_2 \left( U^{(1)}(0^+,t),U^{(2)}(0^+,t) \right) =
\begin{pmatrix}
q_1(0^+,t) + q_2(0^+,t)\\[1mm]
\ds\frac{\gamma C_pR}{\gamma-1}q_2(0^+,t)T_1(0^+,t)
\left(\left(\ds\frac{p_2(0^+,t)}{p_1(0^+,t)}
\right)^{\frac{\gamma-1}{\gamma}}-1\right)\\[2mm]
s_1(0^+,t) - s_2(0^+,t)
\end{pmatrix}
\end{equation}
and the corresponding $\Pi$-functions in (\ref{euler_ccond}),
\begin{equation}
\Pi_1(t) =
\begin{pmatrix}
0\\
H^*(t)\\
0
\end{pmatrix},\quad
\Pi_2(t) =
\begin{pmatrix}
0\\
P^*(t)\\
0
\end{pmatrix}.
\end{equation}

\section{Generalized Riemann Problems for Model Coupling}
\label{sec:GeneralRiemannProblem}
\subsection{Coupling at Junctions}
In this section we will show that the coupling
conditions \eqref{equ:M}, \eqref{equ:H} and \eqref{equ:S} for the network system
\eqref{network_euler_eqs} -- \eqref{network_euler_init} with~\mbox{$G=0$} are
well-defined. Following the theoretical framework applied in
\cite{ColomboGaravello2006,ColomboMauri2008,LangMindt2018}, we
consider a generalised Riemann problem at a junction connecting
pipes with different gas flow models, and show that there
exists a unique self-similar solution in terms of the classical
Lax solution to standard Riemann problems.

Let us denote by $\Omega_j=\{U^{(j)}\in D_{+}^{m_j}\cup D_{-}^{m_j}\}$
for $j=1,\ldots,N$ and $m_j\in\cM$ nonempty sets
and define the overall state space
$\Omega=\Omega_1\times\Omega_2\times\ldots\times\Omega_N$.

\begin{definition}
\label{def:generalizedRiemannProblem}
The generalized Riemann problem at a junction in $x=0$ with~$N$ adjacent
pipes with different flow models is defined through the set of
equations
\begin{equation}
\label{eq:RPNetwork}
\begin{array}{rll}
\partial_tU^{(j)}+\partial_xF_{m_j}(U^{(j)})&=&
0,\;(x,t)\in\R^+\times\R^+,\;m_j\in\cM,\\[2mm]
\Phi\left(U^{(1)}(0^+,t),...,U^{(N)}(0^+,t)\right)&=&\bPi,\\[2mm]
U^{(j)}(x,0)&=&\bU_0^{(j)},\; x\in\R^+,\;j=1,...,N,
\end{array}
\end{equation}
where the states $\bU_0^{(1)},...,\bU_0^{(N)}$ are constant
states in $\Omega$ and $\bPi\in\R^{d}$ is a constant vector
of dimension $d=N+N_o^{\cM_1}$.
\end{definition}

\begin{definition}
\label{def:selfSimilarSolution}
A $\Phi$-solution to the Riemann problem \eqref{eq:RPNetwork} is a self-similar
function $U(x,t):\R^+\times\R^+\rightarrow\Omega$ for which the following hold:
\begin{enumerate}
\item There exists a constant state $U_\ast(\bU_0)=\lim_{x\rightarrow 0^+}U(x,t)$
such that all components $U^{(j)}(x,t)$ coincide with the restriction to
$x>0$ of the Lax solution to the standard Riemann problem for $x\in\R$,
\begin{equation}
\label{eq:RPClassic}
\begin{array}{rll}
\partial_t U^{(j)} + \partial_x F_{m_j}(U^{(j)}) &=&
0,\;(x,t)\in\R\times\R^+,\;m_j\in\cM,\\[2mm]
U^{(j)}(x,0) &=& \left\{
\begin{array}{ll}
\bU^{(j)}_0 & \mbox{if }x>0,\\[1mm]
U^{(j)}_\ast & \mbox{if }x<0.
\end{array}
\right.
\end{array}
\end{equation}
\item The state $U_\ast$ satisfies $\Phi(U_\ast)=\bPi$
for all $t>0$.
\end{enumerate}
\end{definition}

\noindent {\bf Riemann solution for isentropic Euler equations.}
The solution of the standard Riemann problem \eqref{eq:RPClassic}
for the isentropic models $\cM_2$ and $\cM_3$ with initial data~$(U_L,U_R)$
for $x<0$ and $x>0$, respectively, can be described
by a set of elementary waves such as rarefaction and shock waves.

\input{pic1-4Riemann-isentropic.tikz}
%
These waves are parameterisations of the Rankine-Hugoniot jump condition
and the Riemann invariants \cite{LeVeque2002,Toro2009}.
Due to the construction of the network and the subsonic flow
conditions, only $1$-waves can hit the junction, whereas $2$-waves leave
the junction. These waves separate the solution in three states
$(U_L,U_{\ast}^m,U_R)$, $m\in\{\cM_2,\cM_3\}$, see Fig.~\ref{fig:4riemann-polytropic}.
The components of $U_{\ast}^{m}=(\rho_{\ast}^{m},q_{\ast}^{m})$
are determined by the following equations \cite{ColomboGaravello2008}:
\begin{eqnarray}
&&q_{\ast}^{\cM_2}
 = u_L\rho_{\ast}^{\cM_2}-\theta_{2}(\rho_{\ast}^{\cM_2},U_L)
 = u_R\rho_{\ast}^{\cM_2}+\theta_{2}(\rho_{\ast}^{\cM_2},U_R),\\[2mm]
&&q_{\ast}^{\cM_3}
 = q_L-\theta_{3}(\rho_{\ast}^{\cM_3},U_L)
 = q_R+\theta_{3}(\rho_{\ast}^{\cM_3},U_R),
\end{eqnarray}
where
\begin{eqnarray}
&&\theta_2(\rho_{\ast},\bU)=
\begin{cases}
\ds \frac{2\sqrt{\kappa\gamma}}{\gamma-1}\rho_{\ast}
\left(\rho_{\ast}^{\frac{\gamma-1}{2}}-\bar\rho_{\phantom{a}}^{\frac{\gamma-1}{2}}\right)
&\text{if }\rho_{\ast}\leq\bar\rho\text{ (rarefaction)},\\[0.4cm]
\ds \sqrt{\frac{\rho_{\ast}}{\bar\rho}(\rho_{\ast}-\bar\rho)(p_{\ast}-\bar{p})}&
\text{if }\rho_{\ast}>\bar\rho\text{ (shock)},
\end{cases}\\[2mm]
&&\theta_3(\rho_{\ast},\bU)=
\begin{cases}
\ds \frac{2\sqrt{\kappa\gamma}}{\gamma+1}\left(\rho_{\ast}^{\frac{\gamma+1}{2}}-
\bar\rho_{\phantom{a}}^{\frac{\gamma+1}{2}}\right)&\text{if }\rho_{\ast}
\leq\bar\rho\text{ (rarefaction)},\\[0.4cm]
\ds \sqrt{(\rho_{\ast}-\bar\rho)(p_{\ast}-\bar{p})}&
\text{if }\rho_{\ast}>\bar\rho\text{ (shock)}.\\[4mm]
\end{cases}
\end{eqnarray}
\noindent {\bf Riemann solution for polytropic Euler equations.}
For the polytropic Euler equations, the set of waves is extended
by a contact discontinuity which is located between the other two,
see Fig.~\ref{fig:4riemann-polytropic}.
%
\input{pic2-4Riemann-polytropic.tikz}
%
Here, $1$-waves enter the
junction, $2$- and $3$-waves leave the junction while separating
the solution in four states $(U_L,U_{L\ast},U_{R\ast},U_R)$. The
velocity and the pressure are constant across the contact
discontinuity, i.e., we have
\begin{equation}
 p_{\ast}=p_{L\ast}=p_{R\ast}\quad\text{and}\quad u_{\ast}=u_{L\ast}=u_{R\ast}
\end{equation}
The four sought variables $(p_{\ast},u_{\ast},\rho_{L\ast},\rho_{R\ast})$ are
again implicitly defined by means of parameterizations
\cite[Chapt. 4]{Toro2009}, \cite[Chapt. 14.11]{LeVeque2002}. It holds
\begin{eqnarray}
\label{eq:M3up-rel}
 u_{\ast}=u_L-\psi(p_{\ast},U_L)=u_R+\psi(p_{\ast},U_R),\\
 \rho_{L\ast}=\phi(p_{\ast},U_L),\quad\rho_{R\ast}=\phi(p_{\ast},U_R),
\end{eqnarray}
where for $k=L,R$
\begin{eqnarray}
\psi(p_\ast,U_k) &=&
\begin{cases}
  \ds \frac{2c_k}{\gamma-1}\left(\left(
    \frac{p_\ast}{p_k}\right)^{\frac{\gamma-1}{2\gamma}}-1\right)
    &\mbox{if }p_{\ast}\leq p_k \mbox{ (rarefaction)},\\[4mm]
  \ds (p_\ast-p_k)\,\left(\frac{1-\mu^2}{\rho_k(p_\ast+\mu^2p_k)}
    \right)^{\frac{1}{2}}
    &\mbox{if }p_{\ast}>p_k \mbox{ (shock)},
\end{cases}\\[2mm]
\phi(p_\ast,U_k) &=&
\begin{cases}
  \ds \rho_k\,\left(\frac{p_\ast}{p_k}\right)^{\frac{1}{\gamma}}
    &\mbox{if }p_{\ast}\leq p_k \mbox{ (rarefaction)},\\[4mm]
  \ds \rho_k\,\frac{p_\ast+\mu^2p_k}{\mu^2p_\ast+p_k}
    &\mbox{if }p_{\ast}>p_k \mbox{ (shock)},
\end{cases}
\end{eqnarray}
with $\mu^2=(\gamma-1)/(\gamma+1)$ and $c_k^2=\gamma p_k/\rho_k$.
The parameter $p_\ast$ is determined from the second equality
in \eqref{eq:M3up-rel}. The functions $\psi(p_\ast,U_k)$ and
$\phi(p_\ast,U_k)$ are twice continuously differentiable at
$p_\ast=p_k$. The total energy for the inner region is derived
from~$E_{k\ast}=p_\ast/(\gamma-1)+\rho_{k\ast}u_\ast^2/2$ for $k=L,R$.
\abstand

\noindent {\bf Lax curves.} By means of the Riemann solutions, we can
set up the parameterisations of the $k$-waves, the so called $k$-Lax
curves. For the isentropic models, $\cM_2$ and $\cM_3$, they are given by
\begin{align}
&\LX_1^{\cM_2}(\sigma,U_L)=
 \begin{pmatrix}
 \sigma\\
 u_L\sigma-\theta_2(\sigma,U_L)
 \end{pmatrix},\;
\LX_2^{\cM_2}(\sigma,U_R)=
 \begin{pmatrix}
 \sigma\\
 u_R\sigma+\theta_2(\sigma,U_R)
 \end{pmatrix},\\[2mm]
&\LX_1^{\cM_3}(\sigma,U_L)=
 \begin{pmatrix}
 \sigma\\
 q_L-\theta_3(\sigma,U_L)
 \end{pmatrix},\quad
\LX_2^{\cM_3}(\sigma,U_R)=
 \begin{pmatrix}
 \sigma\\
 q_R+\theta_3(\sigma,U_R)
 \end{pmatrix}.
\end{align}
The Lax-curves for the model $\cM_1$ read
\begin{align}
\LX_1^{\cM_1}(\sigma,U_L)&=
 \begin{pmatrix}
 \phi(\sigma,U_L)\\[1mm]
 \phi(\sigma,U_L)(u_L-\psi(\sigma,U_L))\\[1mm]
 \frac{\sigma}{\gamma-1}+\frac{1}{2}\phi(\sigma,U_L)(u_L-\psi(\sigma,U_L))^2
 \end{pmatrix},\\[1mm]
\LX_2^{\cM_1}(\tau,\bar U)&=\bar U+\tau\left(1,\bar u,\frac{1}{2}\bar u^2\right)^T,\\[1mm]
\LX_3^{\cM_1}(\sigma,U_R)&=
 \begin{pmatrix}
 \phi(\sigma,U_R)\\[1mm]
 \phi(\sigma,U_R)(u_R+\psi(\sigma,U_R))\\[1mm]
 \frac{\sigma}{\gamma-1}+\frac{1}{2}\phi(\sigma,U_R)(u_R+\psi(\sigma,U_R))^2
 \end{pmatrix}.
\end{align}
\noindent {\bf Coupling conditions for the $\Phi$-solution.}
Since in our network modelling all pipes are only outgoing from
a junction with respect to $x$-coordinate, the sign of the velocity
in incoming pipes (w.r.t. flow direction) has to be changed when
switching from the standard to the generalised Riemann problem. This
changes the parameterisations of all $\LX_1$-curves. Indeed,
$\LX_1^{m}$ has to be replaced by~$\LX_2^{m}$ for~$m=\cM_2,\cM_3$,
and~$\LX_1^{\cM_1}$ by~$\LX_3^{\cM_1}$. We also note that
a contact discontinuity always travels with positive wave speed.
This is a consequence of constant initial data, the spatial
parametrisation of the pipes and the restriction to subsonic flow.

\input{pic3-4Riemann-coupling.tikz}

The coupling conditions in \eqref{eq:RPNetwork} can now be expressed
in terms of the Lax-curves, see Fig.~\ref{fig:4riemann-coupling} for a schematic
illustration. Let us set $U_{\ast}=U_{L\ast}$ for $\cM_1$.
Then, the sought state $U_\ast$ in Def.~\ref{def:selfSimilarSolution}
satisfies
\begin{equation}
\Phi\left( U^{(1)}_\ast,\ldots,U^{(N)}_\ast\right) = \Pi\in\R^d
\end{equation}
with
\begin{align}
U^{(j)}_{\ast} &= \LX_3^{\cM_1}\left(\sigma_j,\bar U^{(j)}\right),\;
j\in\I_i^{\cM_1},\\
U^{(j)}_{\ast} &= \LX_2^{\cM_1}\left(\tau_j,
\LX_3^{\cM_1}\left(\sigma_j,\bar U^{(j)}\right)\right),\;
j\in\I_o^{\cM_1},\\
U^{(j)}_{\ast} &= \LX_2^{\cM_k}\left(\sigma_j,\bar U^{(j)}\right),\;
j\in\I_i^{\cM_k}\cup\I_o^{\cM_k},\;k=2,3.
\end{align}
Given
constant states $\bU^{(j)}$, mass flux, enthalpy and entropy
can be extracted from~$U^{(j)}_\ast$ using the Lax curves:
\begin{equation}
\label{eq:qhs}
\begin{array}{rlll}
f_j(\sigma_j,\tau_j) &=& f_j(\LX_2^{\cM_1}(
\tau_j,\LX_3^{\cM_1}(\sigma_j,\bU^{(j)}))),
  & j\in\I_o^{\cM_1}, \\[2mm]
f_j(\sigma_j) &=& f_j(\LX_3^{\cM_1}(\sigma_j,\bU^{(j)})),
  & j\in\I_i^{\cM_1},\\[2mm]
f_j(\sigma_j) &=& f_j(\LX_2^{\cM_k}(\sigma_j,\bU^{(j)})),
  & j\in\I_o^{\cM_k}\cup\I_i^{\cM_k},\;k=2,3,
\end{array}
\end{equation}
with $f_j\in\{q_j,h_j,s_j\}$, where $h_j$ and $s_j$ depend
on the model chosen and are given by the expressions
\begin{align}
h_j =&
\begin{cases}
\ds\frac{E_j+p_j}{\rho_j},
  & j\in\I_o^{\cM_1}\cup\I_i^{\cM_1},\\[2mm]
\ds\frac{\kappa_j\gamma}{\gamma-1}\rho_j^{\gamma-1} + \frac{u_j^2}{2},
  & j\in\I_o^{\cM_2}\cup\I_i^{\cM_2},\\[4mm]
\ds\frac{\kappa_j\gamma}{\gamma-1}\rho_j^{\gamma-1},
  & j\in\I_o^{\cM_3}\cup\I_i^{\cM_3},
\end{cases}\\[2mm]
s_j =&
\begin{cases}
\ds c_v\ln \left(\frac{p_j}{\rho_j^\gamma} \right) + s_0,
  & j\in\I_i^{\cM_1}\cup\I_o^{\cM_1},\\[2mm]
\ds s_j, &  j\in\I_i^{\cM_2}\cup\I_i^{\cM_3}.
\end{cases}
\end{align}
Without loss of generality, let $\I_o^{\cM_1}=\{1,2,\ldots,n_0\}$,
$\I_i^{\cM_1}=\{n_0+1,\ldots,n_1\}$, $\I_i^{\cM_2}=\{n_1+1,\ldots,n_2\}$,
and $\I_i^{\cM_3}=\{n_2+1,\ldots,n_3\}$. Accordingly,
the free parameters are $(\sigma_1,\ldots,\sigma_N)$ and
$(\tau_1,\ldots,\tau_{n_0})$.
The coupling conditions \eqref{equ:M}, \eqref{equ:H}, and \eqref{equ:S} can now be
written as
\begin{equation}
\label{eq_coupling}
0 = \Phi(\sigma,\tau) =
\begin{pmatrix}
  \sum_{j=1,\ldots,n_0}\|\nu_j\|\,q_j(\sigma_j,\tau_j)+
  \sum_{j=n_0+1,\ldots,N}\|\nu_j\|\,q_j(\sigma_j)\\[2mm]
  h_{n_0+1}(\sigma_{n_0+1})-h_1(\sigma_{1},\tau_1)\\
  \vdots\\
  h_{n_0+1}(\sigma_{n_0+1})-h_{n_0}(\sigma_{n_0},\tau_{n_0})\\[1mm]
  h_{n_0+1}(\sigma_{n_0+1})-h_{n_0+2}(\sigma_{n_0+2})\\
  \vdots\\
  h_{n_0+1}(\sigma_{n_0+1})-h_{N}(\sigma_{N})\\[1mm]
  s_1(\sigma_1,\tau_1)-s^\ast(\sigma_{n_0+1},\ldots,\sigma_{n_3})\\
  \vdots\\
  s_{n_0}(\sigma_{n_0},\tau_{n_0})-s^\ast(\sigma_{n_0+1},\ldots,\sigma_{n_3})
\end{pmatrix}
\end{equation}
with $s^\ast$ defined through
\begin{equation}
s^\ast(\sigma_{n_0+1},\ldots,\sigma_{n_3}) = \frac{1}{\sum_{j=n_0+1,
\ldots,n_3}\,\|\nu_j\|q_j(\sigma_j)}
\,\sum_{j=n_0+1,\ldots,n_3}\,\|\nu_j\|(q_js_j)(\sigma_j).
\end{equation}
We set $s_j=s^\ast$ for $j\in\I_o^{\cM_2}\cup\I_o^{\cM_3}$, which is
always possible, if $s^\ast$ is well-defined by the
coupling conditions \eqref{eq_coupling}.
The regularity of the Lax curves ensures the property~$\Phi\in C^1(\R^N\times\R^{n_0},\R^d)$.
It remains to show that~\eqref{eq_coupling} has a unique solution. Newton's method
can then be applied to determine the solution vector
$(\sigma^\ast,\tau^\ast)$, which determines the states
$U_\ast^{(j)}$. For the well-posedness of the generalised
Riemann problem~\eqref{eq:RPNetwork} with the coupling
function $\Phi$ defined in~\eqref{eq_coupling}, the following
local result as a generalisation of \cite[Theorem 2.1]{LangMindt2018}
can be given.
\begin{theorem}
\label{th_riemann_junc}
Let $N>\dim(\I_i)>0$ and $\Phi$ as defined in \eqref{eq_coupling}.
Assume constant initial data
$\bU^{(j)}\in D_-^m,\;j\in\I_i^m$, and
$\bU^{(j)}\in D_+^m,\;j\in\I_o^m$, $m\in\{\cM_1,\cM_2,\cM_3\}$,
with~$\Phi(\bU)\!=\!0$ are given. Then there exist positive
constants $\delta$ and $K$ such that for all initial states
$\tU\in\Omega$ with
$\sum_{j=1,\ldots,N}\|\tU^{(j)}-\bU^{(j)}\|\!<\!\delta$,
the Riemann problem~\eqref{eq:RPNetwork} admits
a unique $\Phi$-solution $U(x,t)=\cR^{\Phi}(\tU)$ satisfying
$\Phi(U(0^+,t))\!=\!0$ and
\begin{equation}
\label{grm_lip_init}
\|\cR^{\Phi}(\tU)-\cR^{\Phi}(\bU)\|_{{\bf L}^\infty(\Omega)} \le K\,
\sum_{i=1}^N\|\tilde{U}^{(i)}-\bU^{(i)}\|.
\end{equation}
Additionally, if $\nu$ is replaced by $\hat{\nu}$, where
$\sum_{i=1,\ldots,N}\|\nu_i-\hat{\nu}_i\|\!<\!\delta$, and
$\cR^{\Phi}_{\hat{\nu}}(\tU)$ is the corresponding $\Phi$-solution
for the same initial state $\tU$, then
\begin{equation}
\label{grm_lip_pipe}
\|\cR^{\Phi}_\nu(\tU)-\cR^{\Phi}_{\hat{\nu}}(\tU)\|_{{\bf L}^\infty(\Omega)}
\le K\,\sum_{i=1}^N\|\nu_i-\hat{\nu}_i\|
\end{equation}
with $\cR^{\Phi}_\nu(\tU)\!=\!\cR^{\Phi}(\tU)$.
\end{theorem}
%
\begin{proof}
In the spirit of the implicit function theorem,
it is sufficient to study the determinant of the Jacobian matrix
$D_{(\sigma,\tau)}\Phi(\sigma_0,\tau_0)$ with the two argument vectors
$\sigma_0=(\bp_1,\ldots,\bp_{n_1},\brho_{n_1+1},\ldots,\brho_{n_3})$
and~$\tau_0=0\in\R^{n_0}$. Note that $\Phi(\sigma_0,\tau_0)=0$.
The Jacobian reads
\begin{equation}
\label{cp_jac}
\left(
\begin{array}{ccc|cccc|ccc}
        \hat{q}_{\sigma_1} &
            \!\!\cdots\!\! &
    \hat{q}_{\sigma_{n_0}} &
  \hat{q}_{\sigma_{n_0+1}} &
  \hat{q}_{\sigma_{n_0+2}} &
            \!\!\cdots\!\! &
      \hat{q}_{\sigma_{N}} &
          \hat{q}_{\tau_1} &
            \!\!\cdots\!\! &
        \hat{q}_{\tau_{n_0}}\\[1mm] \hline
       -h_{\sigma_1} & & &
  h_{\sigma_{n_0+1}} & & & &
         -h_{\tau_1} & &\\
 & \!\!\ddots\!\! & & \vdots & & & & & \!\!\ddots\!\! &\\
                     & &
   -h_{\sigma_{n_0}} &
  h_{\sigma_{n_0+1}} & & & & & &
       -h_{\tau_{n_0}}\\[1mm] \hline
                     & & &
  h_{\sigma_{n_0+1}} &
 -h_{\sigma_{n_0+2}} & & & & &\\
  & & & \vdots & & \!\!\ddots\!\! & & \\
                     & & &
  h_{\sigma_{n_0+1}} & & &
     -h_{\sigma_{N}} & & &\\[1mm] \hline
        s_{\sigma_1} & & &
  -s^\ast_{\sigma_{n_0+1}} &
  -s^\ast_{\sigma_{n_0+2}} &
            \!\!\cdots\!\! &
  -s^\ast_{\sigma_{N}} &
            s_{\tau_1} & &\\
 & \!\!\ddots\!\! & & \vdots & & & \vdots & & \!\!\ddots\!\! &\\
                      & &
          s_{\sigma_{n_0}} &
  -s^\ast_{\sigma_{n_0+1}} &
  -s^\ast_{\sigma_{n_0+2}} &
            \!\!\cdots\!\! &
      -s^\ast_{\sigma_{N}} & & &
            s_{\tau_{n_o}}
\end{array}
\right)
\end{equation}
Here, we have used the notations $f_{\mu_i}=\partial_{\mu_i}f_{i}$,
$\hat{q}_{\mu_i}=\|\nu_i\|\partial_{\mu_i}q_{i}$
for $f\!=\!h,s$ and~$\mu=\sigma,\tau$, and
$s^\ast_{\sigma_{i}}=\partial_{\sigma_{i}}s^\ast$. Note that
$s^\ast_{\sigma_j}=0$ for $j>n_3$. From~\eqref{eq:qhs}, we derive the
following derivatives:
\begin{align}
&\partial_{\sigma_j}q_j(\bp_j,0)=\frac{\lambda_3^{\cM_1}}{\bc_j^2},\;
\partial_{\sigma_j}h_j(\bp_j,0)=\frac{\lambda_3^{\cM_1}}{\bc_j\brho_j},\;
\partial_{\sigma_j}s_j(\bp_j,0)=0,\; j\in\I_o^{\cM_1},\\
&\partial_{\tau_j}q_j(\bp_j,0)=\lambda_2^{\cM_1},\;
\partial_{\tau_j}\!h_j(\bp_j,0)=-\frac{\bc_j^2}{(\gamma\!-\!1)\brho_j},\;
\partial_{\tau_j}\!s_j(\bp_j,0)=-\frac{\gamma c_v}{\brho_j},\; j\in\I_o^{\cM_1},\\
&q_j'(\bp_j)=\frac{\lambda_3^{\cM_1}}{\bc_j^2},\;
h_j'(\bp_j)=\frac{\lambda_3^{\cM_1}}{\bc_j\brho_j},\;
\partial_{\sigma_j}s^\ast(\bp,\brho)=
\frac{\|\nu_j\|\lambda_3^{\cM_1}(\bs_j-\bs^\ast)}{\bc^2_j\sum_{j\in\I_i}\|\nu_j\|\bq_j},\;
j\in\I_i^{\cM_1},\\[1mm]
&q_j'(\brho_j)=\lambda_2^{\cM_k},\;h_j'(\brho_j)=\frac{\lambda_2^{\cM_k}\bc_j}{\brho_j},\;
j\in\I_i^{\cM_k}\cup\I_o^{\cM_k},\;k=2,3,\\
&\partial_{\sigma_j}s^\ast(\bp,\brho)=
\frac{\|\nu_j\|\lambda_2^{\cM_k}(\bs_j-\bs^\ast)}{\sum_{j\in\I_i}\|\nu_j\|\bq_j},\;
\; j\in\I_i^{\cM_k},\;k=2,3,
\end{align}
where $\bc_j=\sqrt{\gamma\bp_j/\brho_j}$ for $\cM_1$ and
$\bc_j=\sqrt{\kappa_j\gamma\brho_j^{\gamma-1}}$ for $\cM_2,\cM_3$.
None of the derivatives can vanish, except $\partial_{\sigma_j}s_j$ and
$\partial_{\sigma_j}s^\ast$. Without loss of generality, we number the incoming
pipes in such a way that $\bs_{n_0+1}=\max_{j\in\I_i}\bs_j$.
Then $\bs_{n_0+1}-\bs^\ast\ge 0$, and therefore $s^\ast_{\sigma_{n_0+1}}\le 0$
since $\bq_j<0$ for $j\in\I_i$. Note that $\dim(\I_i)>0$ has
been assumed.
\abstand

\noindent Case 1: $n_0>0$.\\
A closer inspection of the special structure of the Jacobian~\eqref{cp_jac}
reveals that it is regular if and only if all $3\times 3$--matrices
\begin{equation}
D_j =
\begin{pmatrix}
 \hat{q}_{\sigma_j} & \hat{q}_{\sigma_{n_0+1}} & \hat{q}_{\tau_j}\\[1mm]
      -h_{\sigma_j} &       h_{\sigma_{n_0+1}} &      -h_{\tau_j}\\[1mm]
                  0 & -s^\ast_{\sigma_{n_0+1}} & s_{\tau_j}
\end{pmatrix}
\quad\text{ for } j=1,\ldots,n_0,
\end{equation}
are regular. Taking into account the sign of the derivatives, we find
\begin{equation}
\det(D_j) = \hat{q}_{\sigma_j} (h_{\sigma_{n_0+1}}s_{\tau_j}-h_{\tau_j}s^\ast_{\sigma_{n_0+1}})
 + h_{\sigma_j} (\hat{q}_{\sigma_{n_0+1}}s_{\tau_j}+\hat{q}_{\tau_j}s^\ast_{\sigma_{n_0+1}})
 <0.\\[2mm]
\end{equation}
\noindent Case 2: $n_0=0$.\\
In this case, model $\cM_1$ does not appear for outgoing pipes. Hence,
the entropy mix~$s^\ast$ is simply passed as constant entropy value to
the lower order models applied in the outflow region. The parameters
$\tau_0$ disappear and the Jacobian matrix has the simplified form
\begin{equation}
\label{cp_jac_simple}
D_{\sigma}\Phi(\sigma_0) =
\left(
\begin{array}{cccc}
\hat{q}_{\sigma_1} & \hat{q}_{\sigma_2} & \cdots & \hat{q}_{\sigma_N} \\
h_{\sigma_1} & -h_{\sigma_2} && \\
\vdots & & \ddots & \\
h_{\sigma_1} & & & -h_{\sigma_N}
\end{array}
\right).
\end{equation}
We first note that $\hat{q}_{\sigma_j},h_{\sigma_j}\!>\!0$ for all
$j=1,\ldots,N$. Consequently, the column vectors are linearly dependent if
and only if the first vector can be written as the sum of the others. It would
request that $\hat{q}_{\sigma_1} = \alpha_2\hat{q}_{\sigma_2}+\ldots+
\alpha_N\hat{q}_{\sigma_N}$ with $\alpha_j=-h_{\sigma_1}/h_{\sigma_j}$, which
contradicts $\hat{q}_{\sigma_1}>0$. This shows
$\det(D_{\sigma}\Phi(\sigma_0))\ne 0$.

Now, the implicit function theorem ensures the existence of a
$\delta\!>\!0$, a neighbourhood $\cU(v_0)$  of $v_0=(\sigma_0,\tau_0)$
in case 1 or $v_0=\sigma_0$ in case 2, and a function
$\varphi:B(\bar{U},\delta)\rightarrow \cU(v_0)$ such that $\varphi(\bar{U})=v_0$
and $\Phi(v;U)=0$ if and only if $v=\varphi(U)$ for all $U\in B(\bar{U},\delta)$.
The solution $U(x,t)$ can then be identified by the restriction to $x\in\R^+$ of
the solution to the standard Riemann problem~\eqref{eq:RPClassic} with $\bU_0\!=\!\tU$.
The Lipschitz estimate~\eqref{grm_lip_init} follows from the $C^1$-regularity of $\Phi$.
Since $\Phi$ depends smoothly on $\|\nu_i\|$, the same arguments as above can be used to
show~\eqref{grm_lip_pipe}.
\end{proof}

\subsection{Compressor Coupling}
We will now show that the compressor couplings, i.e., the coupling conditions \eqref{equ:CM} and \eqref{equ:CS} accomplished with either \eqref{equ:CP1} or \eqref{equ:CP2}, are also well-defined. The generalized
Riemann problem for the compressor and its self-similar $\Phi$-solution can be formulated
using Definition~\ref{def:selfSimilarSolution} with $N\!=\!2$ and $\Phi$ given by the
corresponding compressor model.\\
We have the following theorem:
\begin{theorem}
\label{th_riemann_comp}
Let $N\!=\!2$, $\|\nu_1\|\!=\!\|\nu_2\|$ and $\Phi$ defined through
\eqref{equ:CM}, \eqref{equ:CS}, and either \eqref{equ:CP1}
or \eqref{equ:CP2}. Assume constant initial data
$\bU^{(1)}\in D_-^i$ and $\bU^{(2)}\in D_+^j$ with
$i,j\in\{\cM_1,\cM_2,\cM_3\}$ and
$\Phi(\bU)\!=\!\bar{\Pi} \!=\!(0,\bar{\Pi}_2,0)^T$ with $\bar{\Pi}_2>0$ are given. Then
there exist positive constants $\delta$ and $K$ such that for all
states $\tU\in\Omega$ with
$\|\tU^{(1)}-\bU^{(1)}\|+\|\tU^{(2)}-\bU^{(2)}\|\!<\!\delta$,
the Riemann problem~\eqref{eq:RPNetwork} admits
a unique $\Phi$-solution $U(x,t)=\cR^{\Phi}(\tU)$ satisfying
$\Phi(U(0^+,t))\!=\!\bar{\Pi}$ and
\begin{equation}
\|\cR^{\Phi}(\tU)-\cR^{\Phi}(\bU)\|_{{\bf L}^\infty(\Omega)} \le K\,
\left(\|\tilde{U}^{(1)}-\bU^{(1)}\|+\|\tilde{U}^{(2)}-\bU^{(2)}\|\right).
\end{equation}
\end{theorem}
%
\begin{proof}
We proceed as in Theorem~\ref{th_riemann_junc}
and study the determinants of the Jacobian matrices. Let us first start
with \eqref{equ:CP1} and the case where model $\cM_1$ is taken at the outflow.
Then the free parameter vector is $(\sigma_1,\sigma_2,\tau_2)$ and the
coupling conditions can be written in the form
$\Phi(\sigma_1,\sigma_2,\tau_2)-\bar{\Pi}=0$. More precisely, we have
\begin{equation}
\label{eq_coupling_comp1}
0 = \hat{\Phi}(\sigma_1,\sigma_2,\tau_2) :=
\begin{pmatrix}
  q_1(\sigma_1)+q_2(\sigma_2,\tau_2)\\[2mm]
  CT_1(\sigma_1)\left(
  \ds\left(\frac{p_2(\sigma_2,\tau_2)}{p_1(\sigma_1)}\right)^
  {\frac{\gamma-1}{\gamma}}-1\right)\\[4mm]
  s_1(\sigma_1)-s_2(\sigma_2,\tau_2)
\end{pmatrix}
-
\begin{pmatrix}
0\\
\bar{H}^\ast\\
0
\end{pmatrix}
\end{equation}
with $C\!=\!R\gamma/(\gamma-1)\!>\!0$. The derivatives taken
at constant state values read
\begin{align}
&\partial_{\sigma_2}q_2(\bp_2,0)=\frac{\lambda_3^{\cM_1}}{\bc_2^2},\;
\partial_{\sigma_2}p_2(\bp_2,0)=1,\;\partial_{\sigma_2}s_2(\bp_2,0)=0,\\
&\partial_{\tau_2}q_2(\bp_2,0)=\lambda_2^{\cM_1},\;
\partial_{\tau_2}p_2(\bp_2,0)=0,\;\partial_{\tau_2}s_2(\bp_2,0)=
-\frac{\gamma c_v}{\brho_2}
\end{align}
and
\begin{align}
&q_1'(\bp_1)=\frac{\lambda_3^{\cM_1}}{\bc_1^2},\;
p_1'(\bp_1)=1,\;T_1'(\bp_1)=\frac{\gamma-1}{\gamma R\brho_1},\;s_1'(\bp_1)=0\;
\text { if }1\in\I_i^{\cM_1},\\
&q_1'(\brho_1)=\lambda_2^{\cM_k},\;
p_1'(\brho_1)=\bc_1^2,\;T_1'(\brho_1)=
\frac{\gamma-1}{R}\kappa_1\brho_1^{\gamma-2},\;s_1'(\brho_1)=0\;
\text { if }1\in\I_i^{\cM_k}
\end{align}
for $k=2,3$. The Jacobian of $\hat{\Phi}$ evaluated at
$(\bp_1,\bp_2,0)$ if $1\in\I_i^{\cM_1}$
or $(\brho_1,\bp_2,0)$ if~$1\in\I_i^{\cM_2}\cup\I_i^{\cM_3}$ has the form
\begin{equation}
D_{(\sigma,\tau)}\hat{\Phi} =
\left(
\begin{array}{ccc}
q_{\sigma_1} & q_{\sigma_2} & q_{\tau_2} \\[1mm]
\partial_{\sigma_1}\hat{\Phi}_2 & \partial_{\sigma_2}\hat{\Phi}_2 &
\partial_{\tau_2}\hat{\Phi}_2\\[1mm]
0 & 0 & s_{\tau_2}
\end{array}
\right),
\end{equation}
where we used the notation $f_{\mu_i}=\partial_{\mu_i}f_i$
for $f=q,s$, and $\mu=\sigma,\tau$. A short calculation of
all derivatives of the second coupling condition $\hat{\Phi}_2$
reveals
\begin{equation}
\partial_{\sigma_1}\hat{\Phi}_2>0,\quad \partial_{\sigma_2}\hat{\Phi}_2<0,
\quad \partial_{\tau_2}\hat{\Phi}_2=0\,.
\end{equation}
Together with $s_{\tau_2}<0$, $q_{\sigma_1}>0$, and $q_{\sigma_2}>0$,
this shows
\begin{equation}
\det(D\hat{\Phi}) = s_{\tau_2} \left( q_{\sigma_1}\,
\partial_{\sigma_2}\hat{\Phi}_2 - q_{\sigma_2}\,\partial_{\sigma_1}\hat{\Phi}_2\right) > 0.
\end{equation}
If one of the models $\cM_2$ or $\cM_3$ is used in the outflow region, the
entropy equality becomes trivial and the coupling conditions reduce to
\begin{equation}
\label{eq_coupling_comp2}
0 = \hat{\Phi}(\sigma_1,\sigma_2) =
\begin{pmatrix}
  q_1(\sigma_1)+q_2(\sigma_2)\\[2mm]
  CT_1(\sigma_1)\left(
  \ds\left(\frac{p_2(\sigma_2)}{p_1(\sigma_1)}\right)^
  {\frac{\gamma-1}{\gamma}}-1\right)
\end{pmatrix}
-
\begin{pmatrix}
0\\
\bar{H}^\ast
\end{pmatrix}.
\end{equation}
Due to $q_2'(\brho_2)\!=\!\lambda_2^{\cM_k}>0$ and
$p_2'(\brho_2)\!=\!\bc_2^2>0$, we finally
conclude that
\begin{equation}
\det(D\hat{\Phi}) = q_{\sigma_1}\,
\partial_{\sigma_2}\hat{\Phi}_2 - q_{\sigma_2}\,\partial_{\sigma_1}\hat{\Phi}_2 < 0.
\end{equation}
It remains to study the case \eqref{equ:CP2}. Again we start with model $\cM_1$ at
the outflow. Then, the second component of the coupling conditions reads
\begin{equation}
\hat{\Phi}_2(\sigma_1,\sigma_2,\tau_2) = Cq_2(\sigma_2,\tau_2)T_1(\sigma_1)
\left( \left( \frac{p_2(\sigma_2,\tau_2)}{p_1(\sigma_1)} \right)^{\frac{\gamma-1}{\gamma}} -1 \right) - \bar{P}^\ast
\end{equation}
with $C=C_pR\gamma/(\gamma-1)>0$. We note $q_2>0$ and have $\bp_2>\bp_1$ due
to the assumption~$\bar{P}^\ast>0$. This gives the inequalities
\begin{equation}
\partial_{\sigma_1}\hat{\Phi}_2>0,\quad \partial_{\sigma_2}\hat{\Phi}_2<0,
\quad \partial_{\tau_2}\hat{\Phi}_2>0,
\end{equation}
and eventually $\det(D\hat{\Phi})>0$. Simplifying the model in the
outflow region to~$\cM_2$ or~$\cM_3$ leads to the coupling conditions
\begin{equation}
\label{eq_coupling_comp3}
0 = \hat{\Phi}(\sigma_1,\sigma_2) =
\begin{pmatrix}
  q_1(\sigma_1)+q_2(\sigma_2)\\[2mm]
  Cq_2(\sigma_2)T_1(\sigma_1)\left(
  \ds\left(\frac{p_2(\sigma_2)}{p_1(\sigma_1)}\right)^
  {\frac{\gamma-1}{\gamma}}-1\right)
\end{pmatrix}
-
\begin{pmatrix}
0\\
\bar{P}^\ast
\end{pmatrix}.
\end{equation}
The same arguments as above show $\det(D\hat{\Phi})<0$.

In all cases, the implicit function theorem guarantees the
existence of a unique $\Phi$-solution to the Riemann problem~\eqref{eq:RPNetwork}
for initial values $\tU$ varying in a small
neighborhood of $\bU$.
\end{proof}

\newpage

\section{The Cauchy Problem at the Junction}\label{sec:CauchyProblemJunction}
We first introduce a few notations.
\begin{definition}
Let
\begin{equation}
\begin{array}{rcll}
\|Y\| &\!=\!& \ds\sum_{i=1}^N \left\| Y^{(i)}\right\| & \mbox{for } Y\in\Omega,\\[5mm]
\|Y\|_{\bfL^1} &\!=\!& \ds\int_{\R^+}\|Y(x)\|\,dx & \mbox{for } Y\in\bfL^1(\R^+;\Omega)\\[5mm]
TV(Y) &\!=\!& \ds\sum_{i=1}^N TV(Y^{(i)}) & \mbox{for } Y\in\bfBV(\R^+;\Omega).
\end{array}
\end{equation}
For the extended variable $\bfq=(U,\Pi)$, a constant state $\bU$ and a constant vector $\bPi$, we consider the metric space
\begin{equation}
\label{def:spaceX}
X=(\bU+\bfL^1(\R^+;\Omega))\times (\bPi+\bfL^1(\R^+;\R^d))
\end{equation}
equipped with the distance and total variation (TV)
\begin{equation}\label{def:d_X}
\begin{array}{rll}
d_X((U,\Pi),(\tilde{U},\tilde{\Pi})) &\!=\!& \| U - \tilde{U} \|_{\bfL^1}
+ \| \Pi - \tilde{\Pi} \|_{\bfL^1}, \\[2mm]
TV(U,\Pi) &\!=\!& TV(U) + TV(\Pi) + \|\Phi(U(0^+)) - \Pi(0+) \|.
\end{array}
\end{equation}
For positive $\delta\in[0,\bdelta]$, we set
$D_\delta(\bfq)\!=\!\{ \bfq\in X:TV(\bfq)\le \delta\}$ and introduce
the set of varying states
$\cU_\delta(\bU)\!=\!\{ U\in\bU+\bfL^1(\R^+;\Omega):TV(U)\le \delta\}$
in the neighborhood of $\bU$.

\end{definition}
Let $G$ denote the vector of the right-hand side functions in~\eqref{network_euler_eqs}
for all pipes and be defined through
\begin{equation}
(G(t,Y))(x) = \left( G_{m_1}(x,t,Y^{(1)}),\ldots,G_{m_N}(x,t,Y^{(N)})\right).
\end{equation}
For the map $G:[0,T]\times \cU_{\bdelta}(\bY)\rightarrow \bfL^1(\R^+;\Omega)$,
we assume that there exist positive constants $L_1$ and $L_2$ such that
for all $t,s\in [0,T]$ the following inequalities are satisfied:
\begin{equation}
\begin{array}{rlll}
\|G(t,Y_1)-G(s,Y_2)\|_{\bfL^1}\! &\!\!\!\le\!\!\!& L_1 \left(\|Y_1\!-\!Y_2\|_{\bfL^1}
\! + |t-s| \right) & \mbox{for all }Y_1,Y_2\in \cU_{\bdelta}(\bY),\\[2mm]
TV(G(t,Y)) &\!\!\!\le\!\!\!& L_2 & \mbox{for all }Y\in \cU_{\bdelta}(\bY).
\end{array}
\end{equation}
This is the usual assumption on $G$, which also covers non-local terms
\cite{ColomboGuerra2007,ColomboGuerra2008} as well as real applications
\cite{ColomboGuerraHertySchleper2009}.

Next we define the Cauchy problem at junctions, which corresponds to our
special set of coupling conditions, and weak solutions.
\begin{definition}
\label{def_sol_cauchy}
Let $N>\dim(\I_i)>0$ and $\Phi$ defined through \eqref{equ:M}, \eqref{equ:H}, \eqref{equ:S}, or~$N\!=\!2$, $\|\nu_1\|\!=\!\|\nu_2\|$ and $\Phi$ defined through \eqref{equ:CM}, \eqref{equ:CS},
and either \eqref{equ:CP1} or~\eqref{equ:CP2}. A weak solution
$U=(U^{(1)},\ldots,U^{(N)})$ on $[0,T]$ to the Cauchy problem
\begin{equation}
\label{prob_cauchy}
\begin{array}{rcll}
\partial_t U^{(i)} + \partial_x F_{m_i}(U^{(i)}) &\!\!=\!\!& G_{m_i}(x,t,U^{(i)}), &
(x,t)\in\R^+\times\R^+,i=1,\ldots,N,\\[2mm]
\Phi(U(0^+,t)) &\!\!=\!\!& \Pi(t),& t\in\R^+,\;\Pi(t)\in\bPi+\bfL^1(\R^+;\R^d)\\[2mm]
U(x,0) &\!\!=\!\!& U_0(x), & x\in \R^+,\;U_0\in\bU+\bfL^1(\R^+;\Omega),
\end{array}
\end{equation}
is a function $U\in\bfC^0([0,T];\bU+\bfL^1(\R^+;\Omega))$
such that $U(t)\in\bfBV(\R^+;\Omega)$ for all \mbox{$t\in [0,T]$}, the initial conditions, $U(x,0)\!=\!U_0(x)$, and the condition at the junction, $\Phi(U(0^+,t))=\Pi(t)$, for a.e. $t>0$ are satisfied. Further,
for all $\varphi\in\bfC^{\infty}_c(\R^+\times (0,T);\R)$ it holds
\begin{equation}
\begin{array}{r}
\ds\int_{0}^{T}\int_{\R^+}\left( U^{(i)}\partial_t\varphi
+ F_{m_i}(U^{(i)})\partial_x\varphi + G_{m_i}(x,t,U^{(i)})\varphi\right)\,dx\,dt
\hspace{1.5cm}\\[2mm]
\ds = \int_{0}^{T}F_{m_i}(U^{(i)}(0^+,t))\,\varphi(0,t)\,dt.
\end{array}
\end{equation}
The weak solution is entropic if for all non-negative $\varphi\in\bfC^{\infty}_c(\Ro^+\times(0,T);\R^+)$
and $i=1,\ldots,N$ it holds
\begin{equation}
\int_0^T \int_{\R^+} \left( \eta_{m_i}\partial_t\varphi +
\psi_{m_i}\partial_x\varphi +
\partial_U\eta_{m_i}G_{m_i}(x,t,U^{(i)})\varphi\right)\,dx\,dt \ge 0.
\end{equation}
for all convex entropy-entropy flux pairs
$(\eta_{m_i}(U^{(i)}),\psi_{m_i}(U^{(i)}))$ of the model
$m_i\in\cM$.
\end{definition}

\subsection{The homogeneous Cauchy problem}
We first start with the homogeneous Cauchy problem, i.e, we set
$G_{m_i}=0$ in
(\ref{prob_cauchy}) for all $i=1,\ldots,N$. The function $\Pi(t)$
is assumed to be constant, $\Pi(t)=\bPi$, to cover the situations
discussed in Theorem~\ref{th_riemann_junc} and~\ref{th_riemann_comp}.
In this case, a solution of the Cauchy problem can be constructed
applying a proper modification of the wave front tracking method by
{\sc Bressan} \cite{Bressan2000} and its natural extension for networks
introduced by {\sc Colombo} and {\sc Mauri} \cite{ColomboMauri2008}. In
terms of the {\it standard Riemann semigroup}, we have the following
\begin{theorem}
\label{th:homog_cauchy}
Let $G_{m_i}\!=\!0$ for all $i=1,\ldots,N$, $\Pi(t)=\bPi$, and the assumptions of
Theorem~\ref{th_riemann_junc} or~\ref{th_riemann_comp} be satisfied
with constant values $\bar{\bfq}=(\bU,\bPi)$. Then there exist positive
constants $\delta$, $K$, a domain $D$, and a semigroup
$S:\R^+\times D\rightarrow D$ such that
\begin{enumerate}[(1)]
\item $\text{cl}_{\bfL_1} \cU_\delta(\bU)\subseteq D$.
\item For all $U\in D$, $S_0(U)=U$ and $S_sS_t(U)=S_{s+t}(U)$.
\item For all $U\in D$, the map $t\rightarrow S_t(U)$ is a weak
entropic solution to the homogeneous Cauchy problem (\ref{prob_cauchy})
in the sense of Definition \ref{def_sol_cauchy}.
\item For $\hU,\tU\in D$ and $s,t\ge 0$ it holds
\[
\|S_t(\hU)-S_s(\tU)\|_{\bfL^1}\le K\,(\|\hU-\tU\|_{\bfL^1}+|s-t|).
\]
\item If $U\in D$ is piecewise constant and $t>0$ sufficiently small,
then $S_t(U)$ coincides with the juxtaposition of the solutions to Riemann
problems centered at the points of jumps or at the junction.
\end{enumerate}
\end{theorem}
\begin{proof}
The proof is based on the wave front tracking algorithm, see \cite{Bressan2000}.
The initial data $U_0$ is approximated by a piecewise constant function $\bU_0$
with a finite number of discontinuities such that $TV(\bU_0)\le\delta$ and
$\| U_0-\bU_0\|_{\bfL^1}\le\varepsilon$. Next, we approximate the local Riemann
problems at each point of a jump in the inner of a pipe by an approximate
Riemann solver - the {\it accurate solver} from \cite[Sect. 7.2]{Bressan2000}.
Rarefaction waves are substituted by rarefaction fans as specified by the
{\it accurate solver}. At junctions, we solve the Riemann problems by means of
our solution procedures introduced above. If $\delta$ is sufficiently small,
Theorems~\ref{th_riemann_junc} and~\ref{th_riemann_comp} ensure the existence and
uniqueness of corresponding solutions. This construction can be continued up to
the first collision among waves in a pipe or until a wave hits the junction. We then
apply the {\it simplified solver} from \cite[Sect. 7.2]{Bressan2000} and its
natural extension to junctions from \cite[Sect. 4.2]{ColomboMauri2008} in order to
keep the total number of waves finite. At a junction, waves with small strength
are reflected into a non-physical wave so that no wave in any other pipe is produced. For $\lim_{k\rightarrow\infty}\varepsilon_k=0$, the algorithm produces
a sequence of $\varepsilon$-solutions $\{U^{\varepsilon_k}\}_{k\in\N}$ in the
sense of {\sc Bressan} \cite[Lemma 7.1]{Bressan2000}. We will now show that
these solutions fulfill {\sc Helly}'s embedding theorem
\cite[Theorem 2.3]{Bressan2000}.

Fix a wave front tracking approximate solution $U^\varepsilon$. Then we have
to show the following three properties:
\begin{equation}
\label{cond_helly}
\begin{array}{rll}
TV(U^\varepsilon(\cdot,t)) &\le& C\quad\text{for all }t\ge 0,\\[1mm]
| U^\varepsilon(x,t) | &\le& M\quad\text{for all }x\in\R^+,\;t\ge 0,\\[1mm]
\int_{\R^+} | U^\varepsilon(x,t) - U^\varepsilon(x,s) |\,dx
&\le& L\,|t-s|\quad\text{for all }s,t\ge 0
\end{array}
\end{equation}
for some $C,M,L>0$. Let $\cJ_i$ and $\cA_i$ denote the set of discontinuities
and the set of approaching wave fronts in the $i$-th pipe for every $t>0$,
respectively, see \cite[Sect. 7.3]{Bressan2000}. The strengths of the waves
of the first, second and third family are denoted by $v_{j,\alpha}$, $j=1,2,3$
and $\alpha\in\cJ_i\cup\cA_i$ with $v=(\sigma,\tau)$ for outgoing waves of the most
accurate model $\cM_1$ and $v=\sigma$ otherwise. The notation $v_\alpha$ is
used when the wave type is not explicitly specified. We now define Glimm-type
functionals
\begin{equation}
\label{def_func_y}
Y(U^\varepsilon(\cdot,t)) = V(U^\varepsilon(\cdot,t)) + \hK_J \, Q(U^\varepsilon(\cdot,t))
\end{equation}
with a suitable constant $\hK_J>0$, which will be specified later,
and the functionals
\begin{equation}
\label{def:interaction-v}
\begin{array}{rll}
V(U^\varepsilon(\cdot,t)) &=& \sum_{i=1}^{N}V_i(U^\varepsilon(\cdot,t)),
\quad V_i = \sum_{\alpha\in\cJ_i}\,C_i(v_\alpha)\,|v_\alpha|,\\[1mm]
Q(U^\varepsilon(\cdot,t)) &=& \sum_{i=1}^{N}Q_i(U^\varepsilon(\cdot,t)),
\quad Q_i = \sum_{(\alpha,\beta)\in\cA_i}\,|v_\alpha|\,|v_\beta|.
\end{array}
\end{equation}
The interaction of waves for all models considered above are well
understood. Estimates for the interaction functional $Q$ are given, e.g., in
\cite[Lemma 4.1]{ColomboGuerra2008} and \cite[Sect. 4.2.]{ColomboMauri2008}.
Here, we focus on the interaction at junctions. In what follows, all wave fronts before interaction time $t$, at which a shock or rarefaction fan hits the junction, are denoted by $v^-$ and wave fronts after interaction time $t$ resulting
from our Riemann solver are denoted by $v^+$. We first recall a simple
consequence of the implicit function theorem applied in the proofs of
Theorems~\ref{th_riemann_junc} and~\ref{th_riemann_comp}. With the function
$\varphi:B(\bU,\delta)\rightarrow\cU(v_0)$, where $\varphi(U)=v$ and $\varphi(\bU)=v_0$,
there exists a constant $C>0$ such that for all $\tU\in B(\bU,\delta)$ with
$\varphi(\tU)=\tilde{v}$ it holds
\begin{equation}
\frac{1}{C}\, |v-\tilde{v}| \le \| U - \tilde{U} \|_{\bfL^1} \le C\,|v-\tilde{v}|.
\end{equation}
Identifying the states before and after interaction by $U^-$ and $U^+$ with
parameters $v^-$ and $v^+$ and setting $v_0=0$, which is always possible by
a change of the coordinate system, we get
\begin{equation}
\frac{1}{C}\,|v^+| \le \| U^+ - \bU \|_{\bfL^1} =
\| \cR^\Phi(U^-) - \cR^\Phi(\bU)\|_{\bfL^1} \le
\tilde{K}\,\| U^- - \bU \|_{\bfL^1} \le \tilde{K} C\,| v^- |
\end{equation}
with a positive constant $\tilde{K}$. From this, we conclude
with the general estimate
\begin{equation}
\label{est:interaction_v}
| v^+ | \le K_J | v^- |\,,\quad\quad K_J=\tilde{K}C^2.
\end{equation}
We are now ready to choose appropriate constants $C_i(v_\alpha)$
in (\ref{def:interaction-v}). We set
\begin{equation}
C_i(v_\alpha) = \left\{
\begin{array}{rll}
1, & i\in\I_o^{\cM_1}, & v_\alpha\text{ belongs to the 2- or 3-family},\\
1, & i\in\I_i^{\cM_1}, & v_\alpha\text{ belongs to the 1-family},\\
2K_J, & i\in\I_o^{\cM_1}, & v_\alpha\text{ belongs to the 1-family},\\
2K_J, & i\in\I_i^{\cM_1}, & v_\alpha\text{ belongs to the 2- or 3-family},\\
1, & i\in\I_o^{\cM_k}, & v_\alpha\text{ belongs to the 2-family, }k=2,3,\\
1, & i\in\I_i^{\cM_k}, & v_\alpha\text{ belongs to the 1-family, }k=2,3,\\
2K_J, & i\in\I_o^{\cM_k}, & v_\alpha\text{ belongs to the 1-family, }k=2,3,\\
2K_J, & i\in\I_i^{\cM_k}, & v_\alpha\text{ belongs to the 2-family, }k=2,3.
\end{array}
\right.
\end{equation}
We will now show the boundedness of the total variation of $U^\varepsilon(x,t)$.
Changes of the functional $V$ as function of $t$ can be expressed in the form
\begin{equation}
\triangle V(t)=\triangle\,V_1(t)+\triangle\,V_2(t)+\triangle\,V_3(t)
\end{equation}
with
\begin{equation}
\begin{array}{rll}
\triangle\,V_1(t) &=&
\ds \sum_{i\in\I_o^{\cM_1}}\sum_{\alpha\in\cJ_i}\left(
|v^+_{2,\alpha}|+|v^+_{3,\alpha}|-2K_J|v^-_{1,\alpha}| \right)\\[6mm]
&& \ds + \sum_{i\in\I_i^{\cM_1}}\sum_{\alpha\in\cJ_i}\left(
|v^+_{1,\alpha}|-2K_J\left(|v^-_{2,\alpha}|+|v^-_{3,\alpha}|\right) \right) \\[6mm]
\triangle\,V_k(t) &=&
\ds \sum_{i\in\I_o^{\cM_k}}\sum_{\alpha\in\cJ_i}\left(
|v^+_{2,\alpha}|-2K_J|v^-_{1,\alpha}| \right)\\[6mm]
&& \ds + \sum_{i\in\I_i^{\cM_k}}\sum_{\alpha\in\cJ_i}\left(
|v^+_{1,\alpha}|-2K_J|v^-_{2,\alpha}|\right),\quad k=2,3.
\end{array}
\end{equation}
Thus, due to estimate (\ref{est:interaction_v}), we have
$\triangle V_k(t)\le -K_JV_k(t-)$, $k=1,2,3$, and therefore
also $\triangle V(t)\le -K_JV(t-)$. Consequently, the map
$t\rightarrow V(t)$ is non-increasing over time at the junction.
For the map $t\rightarrow Q(t)$, we proceed analogously to
\cite{ColomboGaravello2008,ColomboMauri2008} and deduce
\begin{equation}
\ds \triangle Q(t) = Q(t+)-Q(t-)\le Q(t+)-
\sum_{i=1}^{N}\sum_{(\alpha,\beta)\in\cA_i} |v_{\alpha}v_{\beta}|
\le V(t+)^2\le V(t-)^2,
\end{equation}
from which we get the estimate
\begin{equation}
\begin{array}{rll}
\triangle Y(t) &=& \triangle V(t) + \hK_J\triangle Q(t)
\le \left( \hK_JV(t-)-K_J \right) V(t-) \\[2mm]
&\le& \left( \hK_JV(0)-K_J \right) V(t-).
\end{array}
\end{equation}
Let now $\hK_J$ be such that $\hK_JV(0)-K_J<0$.
Then $\triangle Y(t)\le 0$,
showing that the map $t\rightarrow Y(t)$ is also non-increasing.
Observe that the TV-norm is equivalent to $V$, i.e., there exists
a constant $C_1>0$ such that $TV(U)\le C_1\,V(U)$ and
$V(U)\le C_1\,TV(U)$ for all $U\in \cU_\delta(\bU)$. Thus, we can
estimate
\begin{equation}
\begin{array}{rll}
TV(U^\varepsilon(\cdot,t))\le C_1V(U^\varepsilon(\cdot,t)
\le C_1 Y(U^\varepsilon(\cdot,0))
\le C_1(C_1\delta+\hK_JC_1^2\delta^2),
\end{array}
\end{equation}
where we have used that $TV(U^\varepsilon(\cdot,0))\le\delta$.
Boundedness of $U^\varepsilon$ follows from the fact that
all $\varepsilon$-solutions are piecewise constant with a finite
number of discontinuities and have bounded variation. The stability
with respect to time - the third condition in (\ref{cond_helly}) -
has been proven in \cite[Sect. 7]{Bressan2000} independently of the
underlying geometry. As a consequence, an application of {\sc Helly}'s
compactness theorem yields the convergence of a subsequence
$\{U^{\varepsilon_l}\}_{l\in\N}$, where the limit function preserves
the properties of the $\varepsilon$-functions. The limit orbits
$t\rightarrow S_t(U)$ are solutions in the sense of Definition
\ref{def_sol_cauchy}, see \cite[Lemma 7.1]{Bressan2000}. This proves the existence
of a semigroup $S:\R^+\times D\rightarrow D$ and a domain $D$ such that
$\text{cl}_{\bfL^1}\cU_\delta(\bU)\subseteq D$. Hence, statements (1),
(2), (3), and (5) clearly hold.

We pass now to the $\bfL^1$-stability of the $\varepsilon$-solutions with
respect to the initial data, i.e., statement (4). For this aim, we will use
stability functionals from \cite{ColomboHertySachers2008,ColomboMauri2008}.
Let $U^\varepsilon$, $V^\varepsilon\in \cU_\delta(\bU)$ be two piecewise
constant $\varepsilon$-solutions and denote by $q_i^{(l)}(x)$ the shock sizes
implicitly given as solutions of
\begin{equation}
V^\varepsilon (x)= \left\{
\begin{array}{l}
\LX_3^{\cM_1}\left( q_i^{(3)},\LX_2^{\cM_1}
\left( q_i^{(2)},\LX_1^{\cM_1}\left( q_i^{(1)},U^\varepsilon (x)
\right)\right)\right),\; i\in\I^{\cM_1}\\[2mm]
\LX_2^{\cM_k}\left( q_i^{(2)},\LX_1^{\cM_k}
\left( q_i^{(1)},U^\varepsilon (x) \right)\right),\;
k=2,3,\;i\in\I^{\cM_2}\cup\I^{\cM_3},
\end{array}
\right.
\end{equation}
where $\I^{\cM_j}$ is the index set of all pipes with model $\cM_j$
and $\LX_l$, $l=1,2,3$, are the Lax-curves for the shock waves
(see also \cite[(8.4)]{Bressan2000}). For a compact notation, we set
$q_i^{(3)}(x)=0$ for $i\in\I^{\cM_2}\cup\I^{\cM_3}$. Now define the
functional
\begin{equation}
\Theta(U^\varepsilon,V^\varepsilon)=\sum_{i=1}^{N}\sum_{l=1}^{3}
\int_{\R^+} |q_{i}^{(l)}(x)|\, W_{i}^{(l)}(x)\,dx
\end{equation}
with weights
\begin{equation}
\begin{array}{rll}
W_{i}^{(1)}(x) &=& K_s \left( 1 + \kappa_1A_{i}^{(1)} + \kappa_1\kappa_2
\left(Y(U^\varepsilon(\cdot,t)) + Y(V^\varepsilon(\cdot,t)) \right)\right),\\[1mm]
W_{i}^{(j)}(x) &=& 1 + \kappa_1A_{i}^{(j)} + \kappa_1\kappa_2
\left(Y(U^\varepsilon(\cdot,t)) + Y(V^\varepsilon(\cdot,t)) \right),\;j=2,3.
\end{array}
\end{equation}
The constant $K_s$ is defined by
\begin{equation}
K_s = 1 + 2K_JW_{max}\left(\inf_{k=1,2,3}\left|
\lambda_1^{\cM_k}\right|\right)^{-1}\,
\max\left\{\sup\lambda_3^{\cM_1},\sup_{k=2,3}\lambda_2^{\cM_{k}}\right\}
\end{equation}
with $W_{max}=\max\{\max_{i,m_i=1}W_{i}^{(3)},\max_{i,m_i=2,3}W_{i}^{(2)}\}$
and $K_J$ from (\ref{est:interaction_v}). Further, $Y$ is the functional
defined in (\ref{def_func_y}), while the functions $A_{i}^{(j)}$ are defined
in dependence of the models $\cM_{m_i}$ on the $i$-th pipe by $A_{i}^{(j)} := A_{i,j}^{\cM_{m_i}}$ with
\begin{equation}
\begin{array}{rll}
A_{i,j}^{\cM_{2/3}} &=&
\ds\sum \left\{
\left| v_{k_\alpha,\alpha}\right|:
\begin{array}{l}
x_\alpha < x,\;j<k_\alpha\le 2\\
x_\alpha > x,\;1\le k_\alpha < j
\end{array}
\right\}\\[4mm]
&& + \left\{
\begin{array}{ll}
\ds\sum \left\{
\left| v_{j,\alpha}\right|:
\begin{array}{l}
x_\alpha < x,\;\alpha\in J_i(U^\varepsilon)\\
x_\alpha > x,\;\alpha\in J_i(V^\varepsilon)
\end{array}
\right\}
&\,\;\text{if }q_{i}^{(j)}<0,\\[4mm]
\ds\sum \left\{
\left| v_{j,\alpha}\right|:
\begin{array}{l}
x_\alpha < x,\;\alpha\in J_i(V^\varepsilon)\\
x_\alpha > x,\;\alpha\in J_i(U^\varepsilon)
\end{array}
\right\}
&\,\;\text{if }q_{i}^{(j)}\ge 0
\end{array}
\right.
\end{array}
\end{equation}
and
\begin{equation}
\begin{array}{rll}
A_{i,j}^{\cM_{1}} &=&
\ds\sum \left\{
\left| v_{k_\alpha,\alpha}\right|:
\begin{array}{l}
x_\alpha < x,\;j<k_\alpha\le 3,\\
x_\alpha > x,\;1\le k_\alpha <j
\end{array}
\right\}\\[4mm]
&& + \left\{
\begin{array}{ll}
\ds\sum_{j\ne 2} \left\{
\left| v_{j,\alpha}\right|:
\begin{array}{l}
x_\alpha < x,\;\alpha\in J_i(U^\varepsilon)\\
x_\alpha > x,\;\alpha\in J_i(V^\varepsilon)
\end{array}
\right\}
&\,\;\text{if }q_{i}^{(j)}<0,\\[4mm]
\ds\sum_{j\ne 2} \left\{
\left| v_{j,\alpha}\right|:
\begin{array}{l}
x_\alpha < x,\;\alpha\in J_i(V^\varepsilon)\\
x_\alpha > x,\;\alpha\in J_i(U^\varepsilon)
\end{array}
\right\}
&\,\;\text{if }q_{i}^{(j)}\ge 0.
\end{array}
\right.
\end{array}
\end{equation}
These definitions are standard and can be found, e.g., in
\cite[(8.9)]{Bressan2000}, \cite[p. 566]{ColomboMauri2008}, and
\cite{ColomboHertySachers2008}. Now we fix $\kappa_1$ and $\kappa_2$
in such a way that $\delta>0$ in the definition of $\cU_\delta$
can be chosen to uniformly satisfy $1\le W_{i}^{(j)}(x)\le \kappa$
for every $i=1,\ldots,N$ and $j=1,2,3$. Indeed, since $K_s\ge 1$ and
$A_{i}^{(j)}\ge 0$, the lower bound is obvious. By definition, we
have $A_{i}^{(j)}\le \kappa_0 ( Y(U^\varepsilon(\cdot,t))+Y(V^\varepsilon(\cdot,t)))$
with a certain constant $\kappa_0>0$ and consequently $W_{i}^{(j)} \le K (1+2\delta(\kappa_0\kappa_1+\kappa_1\kappa_2))$,
where $K=1$ for $j=2,3$ and $K=K_s$ for $j=1$. Hence, the functional $\Theta$
is equivalent to the $\bfL^1$ distance, i.e.,
\begin{equation}
\frac{1}{C_1} \|U^\varepsilon - V^\varepsilon \|_{\bfL^1}
\le \Theta(U^\varepsilon,V^\varepsilon)
\le C_1\|U^\varepsilon - V^\varepsilon \|_{\bfL^1}
\end{equation}
with a positive constant $C_1$. Applying the same calculations as
in the proof of \cite[Theorem 8.2]{Bressan2000} shows that at any time
$t>0$ when an interaction happens neither in $U^\varepsilon$ nor
$V^\varepsilon$, we have
\begin{equation}
\frac{d}{dt}\Theta(U^\varepsilon(\cdot,t),V^\varepsilon(\cdot,t))
\le C_2 \varepsilon
\end{equation}
with a positive constant $C_2$ that depends only on a bound on the
total variation of the initial data. At an interaction time $t>0$,
we get from above
$\triangle (Y(U^\varepsilon(\cdot,t)) + Y(V^\varepsilon(\cdot,t)))<0$
and hence, by choosing $\kappa_2$ large enough, we obtain
\begin{equation}
\triangle\Theta(U^\varepsilon(\cdot,t),V^\varepsilon(\cdot,t))<0.
\end{equation}
Thus, $\Theta(U^\varepsilon(\cdot,t),V^\varepsilon(\cdot,t))-
\Theta(U^\varepsilon(\cdot,s),V^\varepsilon(\cdot,s))\le C_2(t-s)$
for all $0\le s\le t$. This concludes the proof of statement (4)
by standard arguments given in \cite[Sect. 8.3]{Bressan2000}.
\end{proof}

\subsection{The inhomogeneous Cauchy problem}
Including the source terms $G_{m_i}$ in (\ref{prob_cauchy}) for all
$i=1,\ldots,N$, we have the following result for the
well-posedness of the Cauchy problem:
\begin{theorem}
Let the assumptions of Theorem~\ref{th_riemann_junc}
or~\ref{th_riemann_comp} be satisfied with constant values
$\bar{\bfq}=(\bU,\bPi)$. Let $\cT_t$ be the right
translation defined by $(\cT_t\Pi)(s)=\Pi(t+s)$.
Then there exist positive constants $\delta$, $\delta'$, $K$, domains
$D^t$ for $t\in [0,T]$, and a map~$\cE(s,t_0):D^{t_0}\rightarrow D_\delta$
with $t_0\in[0,T]$ and $s\in [0,T-t_0]$ such that
\begin{enumerate}[(1)]
\item $D_{\delta'}(\bar{\bfq})\subseteq D^t\subseteq D_{\delta}(\bar{\bfq})$
for all $t\in[0,T]$.
\item $\cE(0,t_0)\bfq=\bfq$ for all $t_0\in [0,T],\;\bfq\in D^t$.
\item $\cE(s,t_0)D^{t_0}\subset D^{t_0+s}$ for all
$t_0\in [0,T],\;s\in [0,T-t_0]$.
\item For all $t_0\in [0,T]$, $s_1,s_2\ge 0$
with $s_1+s_2\in [0,T-t_0]$
\[ \cE(s_2,t_0+s_1)\circ\cE(s_1,t_0)=\cE(s_1+s_2,t_0). \]
\item For all $(U_0,\Pi)\in D^{t_0}$, the map
$t\rightarrow\cE(t,t_0)(U_0,\Pi)=(U(t),\cT_t\Pi)$
is the entropic solution to the Cauchy problem~\eqref{prob_cauchy}
in the sense of Definition~\ref{def_sol_cauchy}.
\item For all $t_0\in [0,T]$ and $(U_0,\Pi)\in D^{t_0}$
\[ \lim_{t\rightarrow 0}\frac{1}{t}
\|U(t)-(S_t(U_0,\Pi)+tG(t_0,U_0))\|_{\bfL^1}=0,\]
where $(U(t),\cT_t\Pi)=\cE(t,t_0)(U_0,\Pi)$ and $S_t$ denotes the semigroup generated
from \eqref{prob_cauchy} with $G=0$.
\item For all $t_0\in [0,T],\;s\in [0,T-t_0]$ and
$\bfq,\tilde{\bfq}\in D^{t_0}$
\[ \| \cE(s,t_0)\bfq-\cE(s,t_0)\tilde{\bfq}\|_{\bfL^1} \le K \,\|U-\tU\|_{\bfL^1}
+ K \int_{t_0}^{t_0+s} \| \Pi(t) - \tilde{\Pi}(t)\|_{\bfL^1}\,dt.\]
\end{enumerate}
\end{theorem}
%
\begin{proof}
The proof of this theorem is based on the operator splitting technique
introduced by {\sc Colombo} and {\sc Guerra} \cite{ColomboGuerra2009}
in the framework of differential equations in metric spaces and the
techniques applied by {\sc Colombo}, {\sc Guerra}, {\sc Herty}, and
{\sc Schleper} \cite{ColomboGuerraHertySchleper2009}. In the metric
space $X$ defined in (\ref{def:spaceX}), the following map is considered:
\begin{equation}
F_{t,t_o}(U,\Pi) = \left( S_t(U,\Pi)+tG(t_o,S_t(U,\Pi)),\cT_t\Pi \right).
\end{equation}
It approximates the solution of the inhomogeneous Cauchy problem in the
time interval $[t_o,t_o+t]$. Let $\bfq=(U,\Pi)$, $\tilde{\bfq}=(\tU,\tilde{\Pi})$
and extent the stability function to
\begin{equation}
\Theta(\bfq,\tilde{\bfq})=\sum_{i=1}^{N}\sum_{l=1}^{3}
\int_{\R^+} |q_{i}^{(l)}(x)|\, W_{i}^{(l)}(x)\,dx + \bar{K}
\|\Pi - \tilde{\Pi} \|_{\bfL^1}.
\end{equation}
This extension keeps all properties of the original function $\Theta$
and therefore there exists again a positive constant $C_1$ such that for all
$\bfq$, $\tilde{\bfq}\in D_\delta$ it holds
\begin{equation}
\frac{1}{C_1} d_X(\bfq,\tilde{\bfq})
\le \Theta (\bfq,\tilde{\bfq}) \le C_1 d_X(\bfq,\tilde{\bfq})
\end{equation}
with $d_X$ defined in (\ref{def:d_X}).
Then, the same calculation as in \cite[Lemma 4.7]{ColomboGuerraHertySchleper2009}
yields
\begin{equation}
\Theta(F_{\varepsilon,t_o}\bfq,F_{\varepsilon,t_o}\tilde{\bfq}) \le (1+C\varepsilon)\,\Theta(\bfq,\tilde{\bfq})
\end{equation}
for all $t_o\in [0,T]$, for all $\bfq,\tilde{\bfq}\in D$ (the domain of the
semigroup $S$ from Theorem~\ref{th:homog_cauchy}), $\varepsilon>0$ sufficiently small and a positive constant $C$. This also gives the Lipschitz dependence of $F$ from
$t$ and $\bfq$, see \cite[Proposition 4.8]{ColomboGuerraHertySchleper2009}. Following
\cite[Definition 2.2]{ColomboGuerra2009}, we next construct the Euler $\varepsilon$-polygonal
$F^\varepsilon$ generated by $F$ on the interval $[t_o,t_o+t]$,
\begin{equation}
F^\varepsilon_{t,t_o}\bfq = F_{t-k\varepsilon,t_o+k\varepsilon}\circ
F_{\varepsilon,t_o+k\varepsilon}\circ\ldots\circ F_{\varepsilon,t_o}\bfq,
\end{equation}
where $k=\max\{h\in\N:h\varepsilon<t_o+t\}$. In such a way, the solution
of the inhomogeneous Cauchy problem is approximated by the solution of
the homogeneous Cauchy problem, corrected by an Euler approximation of
the source term after every time step of length $\varepsilon$. It follows
from \cite[Proposition 4.9]{ColomboGuerraHertySchleper2009} that this Euler
approximation is a {\it local flow} in the sense of \cite{ColomboGuerra2009}
and hence allows to apply \cite[Theorem 2.6]{ColomboGuerra2009}. This
ensures the existence of a unique limit semigroup $\cE$ generated by
$F$ and being first order tangent to $F$ in the sense of the tangency
condition (6). With this key observation, the remaining
properties stated in the theorem directly follow from standard results, see,
e.g., the proof of Theorem 2.3 in \cite{ColomboGuerraHertySchleper2009}.
\end{proof}

\section*{Acknowledgments}
This work was supported by the
German Research Foundation within the collaborative research center
TRR154 ``Mathematical Modeling, Simulation and Optimization Using
the Example of Gas Networks'' (DFG-SFB TRR154/2-2018, TP B01).

\bibliographystyle{plain}
\bibliography{bibeuler}
\end{document}

%% file: pic1-4Riemann-isentropic.tikz
\begin{figure}[h!]
    \centering
\subfloat[Case R-R]{
    \begin{tikzpicture}
    \coordinate (o) at (0,0);
    \coordinate (AY) at (0,2);
    \coordinate (AX1) at (-2.5,0);
    \coordinate (AX2) at (2.5,0);

    \draw[->,thick] (o) -- (AY); 
    \draw[->,thick] (AX1) -- (AX2); 
    \draw (AY) node [above]{$t$};
    \draw (AX2) node [right]{$x$};
    \draw (-2,1) node {$U_L$};
    \draw (-0.5,1.5) node {$U_{\ast}^{\cM_k}$};
    \draw (2,1) node {$U_{R}$};
    \draw[-,thick] (o)--(-2,2);
    \draw[-,thick] (o)--(-1.75,2);
    \draw[-,thick] (o)--(-1.5,2);
    \begin{scope}[xscale=-1]
    \draw[-,thick] (o)--(-2.5,2);
    \draw[-,thick] (o)--(-2,2);
    \draw[-,thick] (o)--(-1.75,2);
    \end{scope}
        \end{tikzpicture}
} \hfil
\subfloat[Case S-S]{
    \begin{tikzpicture}
    \coordinate (o) at (0,0);
    \coordinate (AY) at (0,2);
    \coordinate (AX1) at (-2.5,0);
    \coordinate (AX2) at (2.5,0);

    \draw[->,thick] (o) -- (AY); 
    \draw[->,thick] (AX1) -- (AX2); 
    \draw (AY) node [above]{$t$};
    \draw (AX2) node [right]{$x$};
    \draw (-2,1) node {$U_L$};
    \draw (-0.5,1.5) node {$U_{\ast}^{\cM_k}$};
    \draw (2,1) node {$U_{R}$};
    \draw[-,very thick] (o)--(2,2);
    \begin{scope}[xscale=-1]
    \draw[-,very thick] (o)--(2,2);
    \end{scope}
    \end{tikzpicture}
}\\[0.5cm]
%
\subfloat[Case R-S]{
    \begin{tikzpicture}
    \coordinate (o) at (0,0);
    \coordinate (AY) at (0,2);
    \coordinate (AX1) at (-2.5,0);
    \coordinate (AX2) at (2.5,0);

    \draw[->,thick] (o) -- (AY); 
    \draw[->,thick] (AX1) -- (AX2); 
    \draw (AY) node [above]{$t$};
    \draw (AX2) node [right]{$x$};
    \draw[-,thick] (o)--(-2,2);
    \draw[-,thick] (o)--(-1.75,2);
    \draw[-,thick] (o)--(-1.5,2);
    \draw[-,very thick] (o)--(2,2);
    \draw (-2,1) node {$U_L$};
    \draw (-0.5,1.5) node {$U_{\ast}^{\cM_k}$};
    \draw (2,1) node {$U_{R}$};
    \end{tikzpicture}
} \hfil
\subfloat[Case S-R]{
    \begin{tikzpicture}
    \coordinate (o) at (0,0);
    \coordinate (AY) at (0,2);
    \coordinate (AX1) at (-2.5,0);
    \coordinate (AX2) at (2.5,0);

    \draw[->,thick] (o) -- (AY); 
    \draw[->,thick] (AX1) -- (AX2); 
    \draw (AY) node [above]{$t$};
    \draw (AX2) node [right]{$x$};
    \draw (-2,1) node {$U_L$};
    \draw (-0.5,1.5) node {$U_{\ast}^{\cM_k}$};
    \draw (2,1) node {$U_{R}$};
    \begin{scope}[xscale=-1]
    \draw[-,thick] (o)--(-2.5,2);
    \draw[-,thick] (o)--(-2,2);
    \draw[-,thick] (o)--(-1.75,2);
    \draw[-,very thick] (o)--(2,2);
    \end{scope}
    \end{tikzpicture}
}
\caption{Possible wave patterns in the solution of Riemann problems
for the isentropic Euler equations: shock (S) and rarefaction (R).}
\label{fig:4riemann-isentropic}
\end{figure}
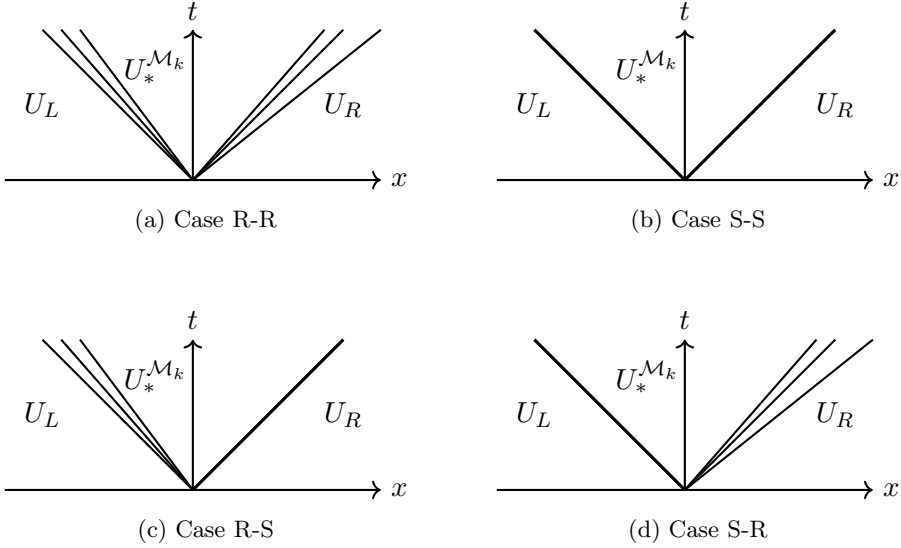


%% file: pic2-4Riemann-polytropic.tikz
%
%
%
%
%
%
%

\begin{figure}[h!]
    \centering
\subfloat[Case R-C-R]{
    \begin{tikzpicture}
    \coordinate (o) at (0,0);
    \coordinate (AY) at (0,2);
    \coordinate (AX1) at (-2.5,0);
    \coordinate (AX2) at (2.5,0);

    \draw[->,thick] (o) -- (AY); 
    \draw[->,thick] (AX1) -- (AX2); 
    \draw (AY) node [above]{$t$};
    \draw (AX2) node [right]{$x$};
    \draw (-2,1) node {$U_L$};
    \draw (-0.5,1.5) node {$U_{L\ast}$};
    \draw (0.8,1.5) node {$U_{R\ast}$};
    \draw (2,1) node {$U_{R}$};
    \draw[-,thick] (o)--(-2,2);
    \draw[-,thick] (o)--(-1.75,2);
    \draw[-,thick] (o)--(-1.5,2);
    \begin{scope}[xscale=-1]
    \draw[-,thick] (o)--(-2.5,2);
    \draw[-,thick] (o)--(-2,2);
    \draw[-,thick] (o)--(-1.75,2);
    \end{scope}
    \draw[-,dashed] (o)--(0.5,2);
    \end{tikzpicture}
} \hfil
\subfloat[Case S-C-S]{
    \begin{tikzpicture}
    \coordinate (o) at (0,0);
    \coordinate (AY) at (0,2);
    \coordinate (AX1) at (-2.5,0);
    \coordinate (AX2) at (2.5,0);

    \draw[->,thick] (o) -- (AY); 
    \draw[->,thick] (AX1) -- (AX2); 
    \draw (AY) node [above]{$t$};
    \draw (AX2) node [right]{$x$};
    \draw (-2,1) node {$U_L$};
    \draw (-0.5,1.5) node {$U_{L\ast}$};
    \draw (0.9,1.5) node {$U_{R\ast}$};
    \draw (2,1) node {$U_{R}$};
    \draw[very thick] (o)--(2,2);
    \begin{scope}[xscale=-1]
    \draw[-,very thick] (o)--(2,2);
    \end{scope}
    \draw[-,dashed] (o)--(0.5,2);
    \end{tikzpicture}
} \\[0.5cm]
%
\subfloat[Case R-C-S]{
    \begin{tikzpicture}
    \coordinate (o) at (0,0);
    \coordinate (AY) at (0,2);
    \coordinate (AX1) at (-2.5,0);
    \coordinate (AX2) at (2.5,0);

    \draw[->,thick] (o) -- (AY); 
    \draw[->,thick] (AX1) -- (AX2); 
    \draw (AY) node [above]{$t$};
    \draw (AX2) node [right]{$x$};
    \draw[-,thick] (o)--(-2,2);
    \draw[-,thick] (o)--(-1.75,2);
    \draw[-,thick] (o)--(-1.5,2);
    \draw[-,very thick] (o)--(2,2);
    \draw[-,dashed] (o)--(0.5,2);
    \draw (-2,1) node {$U_L$};
    \draw (-0.5,1.5) node {$U_{L\ast}$};
    \draw (1,1.5) node {$U_{R\ast}$};
    \draw (2,1) node {$U_{R}$};
    \end{tikzpicture}
} \hfil
\subfloat[Case S-C-R]{
    \begin{tikzpicture}
    \coordinate (o) at (0,0);
    \coordinate (AY) at (0,2);
    \coordinate (AX1) at (-2.5,0);
    \coordinate (AX2) at (2.5,0);

    \draw[->,thick] (o) -- (AY); 
    \draw[->,thick] (AX1) -- (AX2); 
    \draw (AY) node [above]{$t$};
    \draw (AX2) node [right]{$x$};
    \draw (-2,1) node {$U_L$};
    \draw (-0.5,1.5) node {$U_{L\ast}$};
    \draw (0.8,1.5) node {$U_{R\ast}$};
    \draw (2,1) node {$U_{R}$};
    \begin{scope}[xscale=-1]
    \draw[-,thick] (o)--(-2.5,2);
    \draw[-,thick] (o)--(-2,2);
    \draw[-,thick] (o)--(-1.75,2);
    \draw[-,very thick] (o)--(2,2);
    \end{scope}
    \draw[-,dashed] (o)--(0.5,2);
    \end{tikzpicture}
}
\caption{Possible wave patterns in the solution of Riemann problems
for the polytropic Euler equations: shock (S), contact (C), and rarefaction (R).}
\label{fig:4riemann-polytropic}
\end{figure}
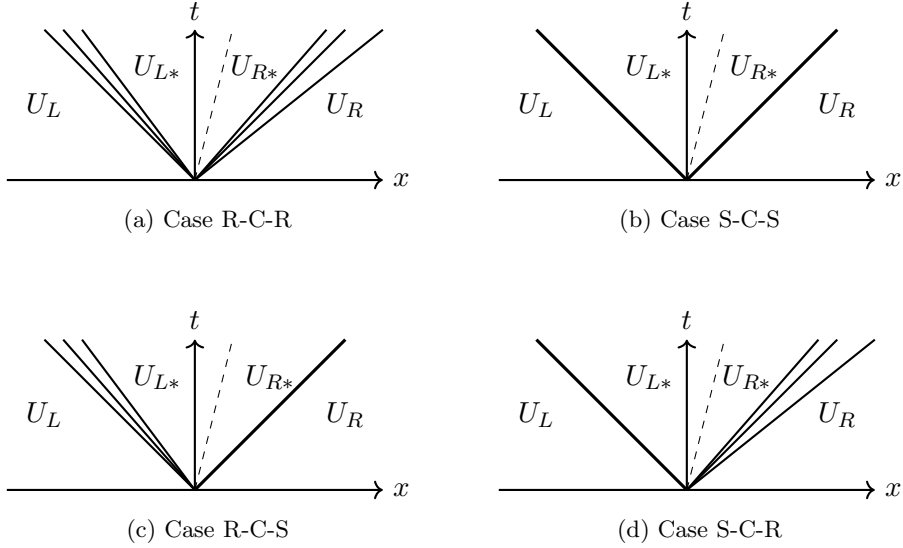


%% file: pic3-4Riemann-coupling.tikz
%
%
%
%
%
%
%

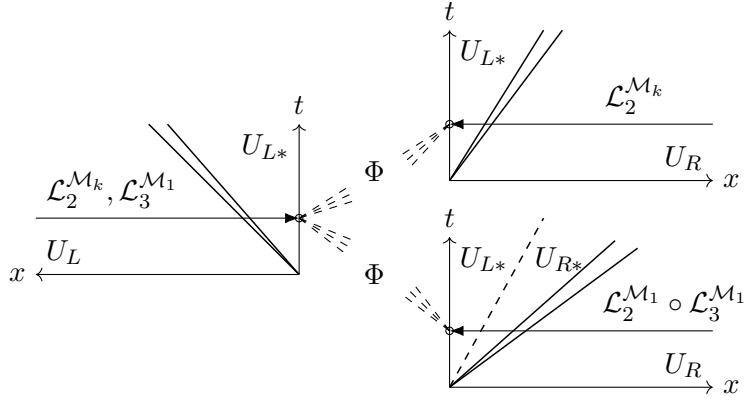
\begin{figure}[h!]
    \centering
  \begin{tikzpicture}
    \coordinate (o1) at (0,0);
    \coordinate (AX1) at (-3.5,0);
    \coordinate (AY1) at (0,2);
    \coordinate (C11) at (-3.5,0.75);
    \coordinate (C12) at (0,0.75);

    \coordinate (o2) at (2,1.25);
    \coordinate (AX2) at (5.5,1.25);
    \coordinate (AY2) at (2,3.25);
    \coordinate (C21) at (2,2);
    \coordinate (C22) at (5.5,2);
    \coordinate (C23) at (5,2);

    \coordinate (o3) at (2,-1.5);
    \coordinate (AX3) at (5.5,-1.5);
    \coordinate (AY3) at (2,0.5);
    \coordinate (C31) at (2,-0.75);
    \coordinate (C32) at (5.5,-0.75);
    \coordinate (C33) at (5,-0.75);

    \draw[<-] (AX1)--(o1);
    \draw[<-] (AY1)--(o1);
    \draw[<-] (AX2)--(o2);
    \draw[<-] (AY2)--(o2);
    \draw[<-] (AX3)--(o3);
    \draw[<-] (AY3)--(o3);
    \draw[-,line width=0.2mm] (o1)--(-2,2);
    \draw[-,line width=0.2mm] (o1)--(-1.75,2);

    \draw[-,line width=0.2mm] (o2)--(3.25,3.25);
    \draw[-,line width=0.2mm] (o2)--(3.5,3.25);

    \draw[-,line width=0.2mm] (o3)--(4.2,0.45);
    \draw[-,line width=0.2mm] (o3)--(4.5,0.35);
    \draw[dashed,line width=0.2mm] (o3)--(3.25,0.75);
    \draw[-{Latex[scale=1.1]}] (C11)--(C12);
    \draw[-{Latex[scale=1.1]}] (C22)--(C21);
    \draw[-{Latex[scale=1.1]}] (C32)--(C31);
    \draw[above right] (AX1) node {$U_L$};
    \draw[below left] (AY1) node {$U_{L\ast}$};
    \draw[left] (AX1) node {$x$};
    \draw[above] (AY1) node {$t$};

    \draw[above left] (AX2) node {$U_R$};
    \draw[below right] (AY2) node {$U_{L\ast}$};
    \draw[right] (AX2) node {$x$};
    \draw[above] (AY2) node {$t$};

    \draw[above left] (AX3) node {$U_R$};
    \draw[below right] (AY3) node {$U_{L\ast}$};
    \draw[below right] (3,0.5) node {$U_{R\ast}$};
    \draw[right] (AX3) node {$x$};
    \draw[above] (AY3) node {$t$};

    \draw[above right] (C11) node {$\LX_2^{\cM_k},\LX_3^{\cM_1}$};
    \draw[above left] (C23) node {$\LX_2^{\cM_k}$};
    \draw[above] (C33) node {$\LX_2^{\cM_1}\circ\LX_3^{\cM_1}$};

    \node[circle,draw,inner sep=1pt,] at (C12){};
    \node[circle,draw,inner sep=1pt,] at (C21){};
    \node[circle,draw,inner sep=1pt,] at (C31){};

%
%


	\node [] at ($(C12)!0.5!(C21)$) {$\Phi$};
		\node [] at ($(C12)!0.5!(C31)$) {$\Phi$};

		\draw[rotate around={17.5:(C12)},dashed] (C12) --(0.75,0.75);
			\draw[rotate around={25:(C12)},dashed] (C12) --(0.75,0.75);
				\draw[rotate around={32.5:(C12)},dashed] (C12) --(0.75,0.75); %
			
		\draw[rotate around={32.5:(C21)},dashed] (C21) --(1.25,2); %
			\draw[rotate around={40:(C21)},dashed] (C21) --(1.25,2);
				\draw[rotate around={47.5:(C21)},dashed] (C21) --(1.25,2);
				
		\draw[rotate around={-22.5:(C12)},dashed] (C12) --(0.75,0.75);
			\draw[rotate around={-30:(C12)},dashed] (C12) --(0.75,0.75);
				\draw[rotate around={-37.5:(C12)},dashed] (C12) --(0.75,0.75);
			
		\draw[rotate around={-37.5:(C31)},dashed] (C31) --(1.25,-0.75);
			\draw[rotate around={-45:(C31)},dashed] (C31) --(1.25,-0.75);
				\draw[rotate around={-52.5:(C31)},dashed] (C31) --(1.25,-0.75);
	
  \end{tikzpicture}
  \caption{Schematic presentation of the coupling with Lax-curves. The
Lax-curve $\LX_2^{\cM_k}$ is used for $k=2,3$.}
\label{fig:4riemann-coupling}
\end{figure}


%% file: ARXIV-V2-MindtLangDomschke2019.bbl
\begin{thebibliography}{10}

\bibitem{BandaHertyKlar2006b}
M.K. Banda, M.~Herty, and A.~Klar.
\newblock Coupling conditions for gas networks governed by the isothermal
  {E}uler equations.
\newblock {\em Netw.~Heterog.~Media}, 1:295--314, 2006.

\bibitem{BandaHertyKlar2006a}
M.K. Banda, M.~Herty, and A.~Klar.
\newblock Gas flow in pipeline networks.
\newblock {\em Netw.~Heterog.~Media}, 1:41--56, 2006.

\bibitem{Bressan2000}
A.~Bressan.
\newblock {\em Hyberbolic Systems of Conservation Laws: The One-dimensional
  Cauchy Problem}, volume~20 of {\em Oxford Lecture Series in Mathematics and
  Its Application}.
\newblock Oxford University Press, 2000.

\bibitem{BrouwerGasserHerty2011}
J.~Brouwer, I.~Gasser, and M.~Herty.
\newblock Gas pipeline models revisited: model hierarchies, nonisothermal
  models, and simulations of networks.
\newblock {\em Multiscale Model. Simul.}, 9:601--623, 2011.

\bibitem{ColomboGaravello2006}
R.M. Colombo and M.~Garavello.
\newblock A well posed {R}iemann problem for the $p$-system at a junction.
\newblock {\em Netw.~Heterog.~Media}, 1:495--511, 2006.

\bibitem{ColomboGaravello2008}
R.M. Colombo and M.~Garavello.
\newblock On the {C}auchy problem for the p-system at a junction.
\newblock {\em SIAM J. Math. Anal.}, 39:1456--1471, 2008.

\bibitem{ColomboGuerra2007}
R.M. Colombo and G.~Guerra.
\newblock Hyperbolic balance laws with a non local source.
\newblock {\em Commun. Partial Differential Equations}, 32:1917--1939, 2007.

\bibitem{ColomboGuerra2008}
R.M. Colombo and G.~Guerra.
\newblock Hyperbolic balance laws with a dissipative non local source.
\newblock {\em Commun. Pure Appl. Anal.}, 7:1077--1090, 2008.

\bibitem{ColomboGuerra2009}
R.M. Colombo and G.~Guerra.
\newblock Differential equations in metric spaces with applications.
\newblock {\em Discrete and Continuous Dynamical Systems, Series A},
  23:733--753, 2009.

\bibitem{ColomboGuerraHertySchleper2009}
R.M. Colombo, G.~Guerra, M.~Herty, and V.~Schleper.
\newblock Optimal control in networks and pipes and canals.
\newblock {\em SIAM J. Control Optim.}, 48:2032--2050, 2009.

\bibitem{ColomboHertySachers2008}
R.M. Colombo, M.~Herty, and V.~Sachers.
\newblock On $2\times 2$ conservation laws at a junction.
\newblock {\em SIAM J. Math. Anal.}, 40:605--622, 2008.

\bibitem{ColomboMauri2008}
R.M. Colombo and C.~Mauri.
\newblock Euler systems for compressible fluids at a junction.
\newblock {\em J.~Hyperbol.~Differ.~Eq.}, 5:547--568, 2008.

\bibitem{DomschkeDuaStolwijkLangMehrmann2018}
P.~Domschke, A.~Dua, J.J. Stolwijk, J.~Lang, and V.~Mehrmann.
\newblock Adaptive refinement strategies for the simulation of gas flow in
  networks using a model hierarchy.
\newblock {\em Electronic Transactions on Numerical Analysis}, 48:97--113,
  2018.

\bibitem{DomschkeKolbLang2011}
P.~Domschke, O.~Kolb, and J.~Lang.
\newblock Adjoint-based control of model and discretization errors for gas flow
  in networks.
\newblock {\em Int. J. Mathematical Modelling and Numerical Optimisation},
  2:175--193, 2011.

\bibitem{DomschkeKolbLang2015}
P.~Domschke, O.~Kolb, and J.~Lang.
\newblock Adjoint-based error control for the simulation and optimization of
  gas and water supply networks.
\newblock {\em Appl. Math. Computat.}, 259:1003--1018, 2015.

\bibitem{EhrhardtSteinbach2005}
K.~Ehrhardt and M.C. Steinbach.
\newblock Nonlinear optimization in gas networks.
\newblock In H.G. Bock, E.~Kostina, H.X. Phu, and R.~Rannacher, editors, {\em
  Modeling, Simulation and Optimization of Complex Processes, Proceedings of
  the International Conference on High Performance Scientific Computing, March
  10-14, 2003, Hanoi, Vietnam}, pages 139--148, 2005.

\bibitem{Herty2008}
M.~Herty.
\newblock Coupling conditions for networked systems of {E}uler equations.
\newblock {\em SIAM J. Sci. Comput.}, 30:1596--1612, 2008.

\bibitem{LangMindt2018}
J.~Lang and P.~Mindt.
\newblock Entropy-preserving coupling conditions for one-dimensional {E}uler
  systems at junctions.
\newblock {\em Netw.~Heterog.~Media}, 13:177--190, 2018.

\bibitem{LeVeque2002}
R.J. LeVeque.
\newblock {\em Finite-Volume Methods for Hyperbolic Problems}.
\newblock Cambridge Texts in Applied Mathematics. Cambridge University Press,
  2002.

\bibitem{Menon2005}
E.S. Menon.
\newblock {\em Gas Pipeline Hydraulics}.
\newblock Taylor \& Francis, 2005.

\bibitem{Osiadacz1996}
A.J. Osiadacz.
\newblock Different transient flow models – limitations, advantages, and
  disadvantages.
\newblock in 28th Annual Meeting of PSIG San Francisco, 9606 PSIG Conference
  Paper, OnePetro, Richardson, 1996.

\bibitem{OsiadaczChaczykowski2001}
A.J. Osiadacz and M.~Chaczykowski.
\newblock Comparison of isothermal and non-isothermal pipeline gas flow models.
\newblock {\em Chem. Engrg. J.}, 81:41--51, 2001.

\bibitem{Reigstad2014}
G.A. Reigstad.
\newblock Numerical network models and entropy principles for isothermal
  junction flow.
\newblock {\em Netw.~Heterog.~Media}, 9:65--95, 2014.

\bibitem{Reigstad2015}
G.A. Reigstad.
\newblock Existence and uniqueness of solutions to the generalized {R}iemann
  problem for isentropic flow.
\newblock {\em SIAM J. Appl. Math.}, 75:679--702, 2015.

\bibitem{Toro2009}
T.~Toro.
\newblock {\em Riemann solvers and numerical methods for fluid dynamics: a
  practical introduction}.
\newblock Springer, 2009.

\end{thebibliography}
